\documentclass[a4paper,11pt]{amsart}

\usepackage{amssymb}
\usepackage{amsmath,amsthm}
\usepackage{enumerate,paralist}
\usepackage{amstext}
\usepackage{dsfont}
\usepackage[english]{babel}
\usepackage{color}
\usepackage{textgreek}

\usepackage{hyperref}
\usepackage[a4paper, left=2.5cm, right=2.5cm, top=3cm, bottom=3cm]{geometry}
\usepackage{mathtools}
\usepackage{latexsym}
\usepackage{mathrsfs}
\usepackage{MnSymbol}
\usepackage{bbm}
\usepackage{enumitem}

\theoremstyle{plain}
\newtheorem{theorem}{Theorem}
\newtheorem{corollary}{Corollary}
\newtheorem{lemma}{Lemma}
\newtheorem{proposition}{Proposition}
\theoremstyle{definition}

\newtheorem{remark}{Remark}
\newtheorem*{assumption}{Assumptions}

\numberwithin{theorem}{section}
\numberwithin{corollary}{section}
\numberwithin{lemma}{section}
\numberwithin{definition}{section}
\numberwithin{example}{section}
\numberwithin{remark}{section}
\numberwithin{proposition}{section}

\usepackage{color}
\definecolor{MY}{rgb}{0.5,0,0.45}

\newcommand{\1}{\mathds 1}

\newcommand{\eps}{\varepsilon}

\newcommand{\dD}{\mathcal{D}}
\newcommand{\eE}{\mathcal{E}}

\newcommand{\ent}{\mathrm{Ent}}
\newcommand{\var}{\mathrm{Var}}

\def\al{\alpha}
\def\be{\beta}
\def\la{\lambda}

\def\del{\delta}
\def\Del{\Delta}

\def\si{\sigma}

\def\eps{\varepsilon}

\def\Om{\Omega}
\def\dOm{\partial \Omega}

\def\pd{\partial}

\def\laa{\lambda_\alpha}
\def\sib{\sigma_\beta}

\def\lp{\left(}
\def\rp{\right)}

\begin{document}

\title[]{Functional inequalities for doubly weighted Brownian motion with sticky-reflecting boundary diffusion}
\author{Marie Bormann}
\address{Universit\"at Leipzig, Fakult\"at f\"ur Mathematik und Informatik, Augustusplatz 10, 04109 Leipzig, Germany and Max Planck Institute for Mathematics in the Sciences, 04103 Leipzig, Germany}
\email{bormann@math.uni-leipzig.de}

\begin{abstract}
We give upper bounds for the Poincar\'{e} and logarithmic Sobolev constants for doubly weighted Brownian motion on manifolds with sticky-reflecting boundary diffusion under curvature assumptions on the manifold and its boundary. To achieve this we use an interpolation approach based on energy interactions between the boundary and the interior of the manifold as well as the weighted Reilly formula. Along the way we also obtain a lower bound on the first nontrivial doubly weighted Steklov eigenvalue and an upper bound on the norm of the doubly weighted boundary trace operator on Sobolev functions. We also consider the case of doubly weighted Brownian motion with pure sticky reflection.
\end{abstract}

\date{\today}

\maketitle

\section{Introduction}

Let $\Om$ be a smooth compact connected Riemannian manifold of dimension $d\ge2$ with smooth connected boundary $\dOm$. We consider the semigroup on $C(\Om)$ induced from the Feller generator $(\dD(L),L)$ given by

\begin{align*}
	\dD(L) &= \{f\in C(\Om)\ |\ L f\in C(\Om)\}, \\
	L f &= \lp \Delta f + \frac{\nabla \al}{\al}\cdot \nabla f \rp \1_{\Om^\circ} - \lp\frac{\al}{\be} \frac{\partial f}{\partial N}\rp \1_{\dOm}  + \del \lp \Delta^\tau f + \frac{\nabla^\tau \be}{\be}\cdot \nabla^\tau f\rp \1_{\dOm},
\end{align*}
where $\frac{\pd}{\pd N}$ is the derivative in the direction of the outward pointing unit normal vector field~$N$, $\nabla^\tau$ is the tangential gradient on $\dOm$, $\Del^\tau$ is the Laplace-Beltrami operator on $\dOm$, $\al,\be$ are appropriate weight functions on $\Om$ and $\delta \geq 0$ is a fixed parameter. Note that by $C(\Om)$ we always mean $C\lp\overline{\Om}\rp$. The induced Markov process $(X_t)_{t\ge0}$ is a diffusion on $\Om$ which performs Brownian motion with drift $\nabla\al/\al$ in the interior $\Om^\circ$ and whose boundary behaviour consists of Brownian diffusion with drift $\nabla^\tau \be/\be$ according to speed $\delta$ along the boundary $\dOm$ as well as sticky reflection back into the interior with intensity depending on the quotient $\al/\be$. 
We make a distinction between the cases $\delta>0$ and $\delta=0$ and refer to the former as Brownian motion with sticky-reflecting boundary diffusion doubly weighted according to $\al,\be$ and to the latter as Brownian motion with purely sticky reflection without boundary diffusion doubly weighted according to $\al,\be$. Note in particular that we allow different weight functions $\al$ for the interior and $\beta$ for the boundary. The non-weighted case is included by choosing $\al,\be$ to be constant. The boundary behaviour is called sticky as the process sojourns at the boundary and thus fundamentally differs from diffusion reflected or killed at the boundary. In general, sticky reflection from the boundary arises from reflection by a time change at the boundary which causes the process to slow down at the boundary while not changing its behaviour away from the boundary, see~\cite{MR154338, MR3271518}. The stickiness is independent of the value of $\delta$ and may be formalised in terms of a positive occupation time on the boundary (\cite{gv}). 
However, while the occupation time on the boundary is positive, the process never spends a whole interval of time on the boundary, see~\cite{MR611513,MR4064533}. Overall, this boundary behaviour is `non-standard' as opposed to the more `standard' reflection or killing at the boundary corresponding to Neumann or Dirichlet boundary conditions. As a boundary condition in the present setting we may write
\begin{equation*}
\Delta f + \frac{\nabla \al}{\al}\cdot \nabla f = - \frac{\al}{\be} \frac{\partial f}{\partial N} +\delta\lp\Delta^\tau f + \frac{\nabla^\tau \be}{\be}\cdot \nabla^\tau f\rp  \text { on } \dOm, 
\end{equation*}
which arises from requiring $L f$ to be continuous at the boundary.\\
The operator $L$ is symmetric with respect to the measure $\mu:=\al d\la + \be d\si$, where $\la$ and $\si$ are respectively the volume measure on $\Om$ and the Hausdorff measure on $\dOm$. The associated Dirichlet form is 
\begin{equation*}
\eE(f,g):= \int \nabla f \cdot \nabla g\ \al d\la + \delta \int \nabla^\tau f \cdot \nabla^\tau g\ \be d \si \text{ in } L^2(\Om,\mu)
\end{equation*} 
with domain $\dD(\eE)$ equal to the closure of $C^1(\Om)$ wrt.\ the norm defined by $\lp\eE(\cdot)+ |\cdot|_{2,\mu}^2\rp^{1/2}$.
Furthermore, the Markov process induced by $(\dD(L),L)$ may be identified as a weak solution of the following stochastic differential equation
\begin{align*}
	dX_t &= \1_{\Om^\circ} (X_t) \lp \sqrt{2}dB_t +\frac{\nabla \al}{\al}(X_t)dt \rp -\1_{\dOm}(X_t)\frac{\al}{\be}(X_t)N(X_t)dt \notag \\
		&+ \del \1_{\dOm}(X_t)\lp \sqrt{2}dB^\tau_t + \frac {\nabla^\tau \be}{\be}(X_t)dt \rp, \\
	dB^\tau_t &= P(X_t)\circ dB_t,\notag \\
	X_0&=x, \notag
\end{align*}
for q.e.\ $x\in\Om$ where $P$ is the projection on the tangent space, see~\cite{gv}. Strong solutions are not expected to exist due to the fact that in~\cite{MR3271518} it was shown that Brownian motion on $\mathbb{R}_+$ with sticky reflection at zero arises from a stochastic differential equation which does not have strong solutions. 
In the case $\delta=0$ without boundary diffusion the infinitesimal generator reduces to $(\hat{L},D(\hat{L}))$,
\begin{align*}
	\dD(\hat{L}) &= \{f\in C(\Om)\ |\ \hat{L} f\in C(\Om)\}, \\
	\hat{L} f &= \1_{\Om^\circ} \lp \Delta f + \frac{\nabla \al}{\al}\cdot \nabla f \rp - \1_{\dOm} \frac{\al}{\be} \frac{\partial f}{\partial N}.
\end{align*}
The operator $\hat{L}$ is symmetric with respect to the measure $\mu$, the associated Dirichlet form is 
\begin{equation*}
\hat{\eE}(f,g):= \int \nabla f \cdot \nabla g\ \al d\la  \text{ in } L^2(\Om,\mu),
\end{equation*} 
with domain $\dD(\hat{\eE})$ equal to the closure of $C^1(\Om)$ wrt.\ the norm defined by $\lp\hat{\eE}(\cdot)+ |\cdot|_{2,\mu}^2\rp^{1/2}$
and the corresponding stochastic differential equation is
\begin{align*}
	d\hat{X}_t &= \1_{\Om^\circ} (\hat{X}_t) \lp \sqrt{2}dB_t +\frac{\nabla \al}{\al}(\hat{X}_t)dt \rp -\1_{\dOm}(\hat{X}_t)\frac{\al}{\be}(\hat{X}_t)N(\hat{X}_t)dt, \notag \\
	\hat{X}_0&=x \notag.
\end{align*}

Brownian motion with sticky-reflecting boundary diffusion has been studied as early as the 1950s, e.g.\ in~\cite{wentzell} which is concerned with finding the most general boundary condition for an elliptic operator which still produces an infinitesimal generator of a Markov semigroup. First rigorous constructions date back to~\cite{MR126883,MR287612,MR929208} for special domains $\Om$, see also~\cite{MR611513,MR1011252} and the references therein for more details including constructions of such diffusion processes by means of stochastic differential equations. Later extensions to jump-diffusion processes on general domains followed, cf.~\cite{MR245085,MR4176673}. A derivation and physical interpretation as well as a discussion of qualitative regularity properties of the associated semigroups have been given in~\cite{MR2215623, MR4065110}. A rigorous construction of doubly weighted sticky-reflecting Brownian motion with or without boundary diffusion in terms of Dirichlet forms has been given recently in~\cite{gv}. Furthermore, interest for the case including boundary diffusion has arisen lately in the context of interacting particle systems with mean-field or zero-range pair interaction~\cite{MR4096131,MR4704529}. Finally, large deviations for Brownian motion with sticky-reflecting boundary diffusion have been considered recently in~\cite{casteras2025largedeviationsstickyreflectingbrownian}, the metric gradient flow structure of the Fokker-Planck equation associated with Brownian motion with sticky-reflecting boundary diffusion has been investigated in~\cite{MR4901547,bormann2025gradientflownoneheat} and isoperimetric inequalities have been considered in~\cite{bormann2025cheegertypeinequalitydriftlaplacian}.\\

Here we show that Poincar\'{e} and logarithmic Sobolev inequalities are fulfilled for Brownian motion with sticky reflection and with or without boundary diffusion doubly weighted by $\al,\be$ and give upper bounds for the Poincar\'{e} and logarithmic Sobolev constants. The bounds are in terms of the geometry of the manifold and the given weights as well as the more `standard' Poincar\'{e} and logarithmic Sobolev constants for simple reflected Brownian motion on $\Om$ or Brownian motion on $\dOm$. This allows to describe the rate of convergence to equilibrium characterised by the invariant measure $\mu$. While doing so we also manage to give a lower bound on the first nontrivial doubly weighted Steklov eigenvalue and an upper bound on the norm of the Sobolev trace operator corresponding to the weights $\al,\be$. For simplicity we will restrict to the cases $\delta=1$ and $\delta=0$, since the case $\delta\in(0,\infty)$ is a simple adaptation of the former.  

For constant weights $\al,\be$ functional inequalities have been considered in the setting of positive curvature in~\cite{mvr} and for general curvature bounds in~\cite{brw}. The contribution of this note is to transfer these known results to the case of a more general class of weights by using again some of the essential techniques from~\cite{brw}. Similar results for the weighted Laplacian with Dirichlet, Neumann or Robin boundary condition have been shown in~\cite{MR1262968,MR1432586,setti,MR3951758} and many other articles. Note that the setting introduced above arises from the unweighted case by adding the same weights both to the Dirichlet form and to the Hilbert space. Compared to the nonweighted case this means that additional drift terms in the interior and on the boundary are included. Similarly just adding weights to the Hilbert space but not the form would describe a time change for the corresponding stochastic process, and just adding weights to the form but not the Hilbert space changes the way energy is measured and can for example be used when considering heterogeneous media. See~\cite{MR2778606} for more details.\\

In the remainder of this section we introduce the notation and underlying concepts as well as assumptions on the weight functions. Subsequently in section~\ref{sec:PI} the Poincar\'{e} inequality is discussed firstly in the case including boundary diffusion for coinciding weights (subsection~\ref{ssec:coinweights}) or general weights (subsection~\ref{ssec:genweights}) where a lower bound on the first nontrivial doubly weighted Steklov eigenvalue is a byproduct, and also in the case without boundary diffusion (subsection~\ref{ssec:pcwobd}). In section~\ref{sec:LSI} the logarithmic Sobolev inequality is discussed firstly in the case including boundary diffusion (subsection~\ref{ssec:lsiwbd}) and also in the case without boundary diffusion (subsection~\ref{ssec:lsiwobd}). Explicit upper bounds for the logarithmic Sobolev constants in both cases are finally obtained in section~\ref{sec:bii} where an upper bound on the norm of the doubly weighted boundary trace operator on Sobolev functions is a byproduct.

\subsection{Notation and underlying concepts}

In this preliminary section the notation and underlying concepts used throughout the paper will be introduced.
By $\Omega$ we denote a smooth compact connected Riemannian manifold of dimension $d \ge 2$ with smooth connected boundary $\dOm$. In particular, when writing $\Om$ the boundary is included. We write $\lambda$ for the volume measure on $\Om$ and $\sigma$ for the Hausdorff measure on $\dOm$, where for simplicity $|\Om|:=\lambda(\Om)$ and $|\dOm|:=\sigma(\dOm)$. 
By $C, C^1,C^\infty, C^\infty_0$ we denote respectively the spaces of real-valued functions that are continuous, once continuously differentiable, smooth, smooth and vanishing at infinity and denote in parentheses whether functions on $\Om$ or $\dOm$ are considered. We will consider weight functions $\al:\Om\to\mathbb{R}$ for $\lambda$ and $\beta:\Om\to\mathbb{R}$ for $\sigma$ and describe the normalised weighted measures via $ \laa:=\frac{\al}{A}d\la,\sib:= \frac{\be}{B}d\si$ with $A:=\int_\Om\al d\la, B:=\int_{\dOm}\be d\si$. Furthermore, we write $d\mu:=\al d\la+\be d\si$. Note that $\be$ is defined on the whole $\Om$ and not just on $\dOm$.
$W^{1,q}$ and $L^p$ denote Sobolev and Lebesgue spaces and we indicate in parentheses the relevant space and weighted measure. We write $|\cdot|$ for the norm induced by the Riemannian metric, $\|\cdot\|$ for the Hilbert-Schmidt norm, $|\cdot|_\infty$ for the $L^\infty$ norm wrt.\ $\la$ on $\Om$ and $|\cdot|_p$ for $L^p$ norms, where a further lower index is added to indicate the relevant weighted measure.
We write $\vert_{\dOm}$ for the trace operator from $W^{1,2}(\Om,\al d\la)$ to $L^2(\dOm,\be d\si)$.
$\nabla, \nabla^2, \Delta, \Delta_\al$ denote the gradient, Hessian, Laplace-Beltrami and $\al$-weighted Laplacian $\Delta+ \frac{\nabla \al}{\al} \cdot \nabla f$, and $\nabla^\tau,\Delta^\tau$ denote the gradient and Laplace-Beltrami operator on $\dOm$. We write $N$ for the outward-pointing normal vector field on $\dOm$. Furthermore, we write $\mathrm{sect}, \mathrm{Ric}, \mathrm{Ric}_\al, \mathrm{R}_{\al,n}$ for the sectional curvature, Ricci curvature, weighted Ricci curvature $\mathrm{Ric}-\nabla^2(\ln(\al))$ and $n$-weighted Ricci curvature $\mathrm{Ric} - \nabla^2(\ln(\al))- \frac{1}{n-d}d(\ln(\al)) \otimes d(\ln(\al))$ for $n\in[d,\infty]$ as well as $\mathrm{I\!I}$ for the second fundamental form, $H$ for the mean curvature and $H_\al$ for the $\al$-weighted mean curvature $H+\frac{1}{\al} \frac{\partial \al}{\partial N}$ on $\dOm$. For a real-valued function $f$ on $\Om$ we denote by $(f)^-$ and $(f)^+$ respectively the negative and positive part of $f$ and write $\overline{f}=\frac{1}{|\Om|}\int_\Om f d\la$ for its mean value on $\Om$. We define the distance to the boundary function $\rho_{\dOm}:\Om\to\mathbb{R}, \rho_{\dOm}(x):=dist(x,\dOm)$, with distances measured as induced by the Riemannian metric. $\rho_{\dOm}$ is 1-Lipschitz, and in a sufficiently small neighbourhood of the boundary it is smooth. The set where smoothness fails is the so-called cut locus of the boundary, see e.g.~\cite{MR3469435} for more details. Moreover, for a function $h:\mathbb{R}_+\to\mathbb{R}$ we write $h^{-1}(0)$ for the first nonnegative zero of $h$, i.e.\ $h^{-1}(0):=\inf\{t\ge 0: h(t)=0\}$, where $h^{-1}(0)=\infty$ if $h(t)>0$ for all $t\ge 0$, and $h'$ for the first derivative of $h$. We write $C_{p,q}$ for the Sobolev-Poincar\'{e} constants on $\Om$ and $C^{\be,\al}_{p,q}$ for the doubly weighted analogue, characterised respectively as the smallest constants fulfilling 
\begin{align*}
\lp \int_\Om |f-\overline{f}|^p d\lambda\rp^{1/p} &\le C_{p,q} \lp \int_\Om |\nabla f|^q d\la\rp^{1/q} \ \forall f\in W^{1,q}(\Om),\\
\lp \int_\Om |f-\overline{f}|^p \frac{\be}{B}d\lambda\rp^{1/p} &\le C^{\be,\al}_{p,q} \lp \int_\Om |\nabla f|^q \frac{\al}{A} d\la \rp^{1/q} \ \forall f\in W^{1,q}(\Om,\al d\la),
\end{align*}
for $q\in[1,d),p\in[1,\frac{qd}{d-q}]$. The unweighted Sobolev-Poincar\'{e} inequality is known to hold on a smooth, compact $d$-dimensional Riemannian manifold (see e.g.\ \cite{MR1688256}) and analogously the weighted version follows if we assume that $\al,\be$ are such that $W^{1,q}(\Om,\al d\la)$ embeds into $L^p(\Om,\be d\la)$ and $W^{1,q}(\Om,\al d\la)$ embeds compactly into $L^q(\Om,\al d\la)$. Both of these conditions are fulfilled if $|\be|_\infty,\left|\frac{1}{\al}\right|_\infty<\infty$. Furthermore, weighted Sobolev spaces and Sobolev embeddings have been studied for specific types of weights e.g.\ in~\cite{MR2424078} and~\cite{MR664599}. The latter treats the case of weights that are functions of the distance to the boundary $\rho_{\dOm}$. An elementary computation allows to give an upper bound on $C_{p,q}^{\be,\al}$ in terms of $C_{p,q}$ in case $|\be|_\infty,\left|\frac{1}{\al}\right|_\infty<\infty$:
\begin{equation*}
C_{p,q}^{\be,\al} \le  \lp \frac{|\be|_\infty}{B}\rp^{1/p} \lp \left|\frac{A}{\al}\right|_\infty\rp^{1/q} C_{p,q}.
\end{equation*}
\smallskip\\
In this paper we study Poincar\'{e} and logarithmic Sobolev inequalities for Brownian motion on manifolds with `non-standard' boundary behaviour. We briefly describe these two functional inequalities, see e.g.~\cite{bgl} and the references therein for more details. In general, a Markov process corresponding to a Dirichlet form $(\dD(\tilde{\eE}),\tilde{\eE})$ is said to satisfy a Poincar\'{e} inequality with respect to an invariant measure $\nu$ if
there is a constant $C$ such that
\begin{equation*}
\mathrm{Var}_\nu (f) := \int f^2 d\nu - \lp \int f d\nu \rp^2 \le C \tilde{\eE}(f)
\end{equation*}
for all $f\in\dD(\tilde{\eE})$ or for all $f$ in a subset that is dense in $\dD(\tilde{\eE})$ wrt.\ the norm defined by $\lp\tilde{\eE}(\cdot)+ |\cdot|_{2,\nu}^2\rp^{1/2}$.
This is equivalent to the following exponential decay in variance property
\begin{equation*}
\mathrm{Var}_\nu(P_t f) = |P_tf - \nu(f)|_{2,\nu}^2 \le e^{-2t/C} \mathrm{Var}_\nu(f)\ \forall f\in L^2(\nu), t\ge0,
\end{equation*}
where by $(P_t)_{t\ge0}$ we denote the corresponding Markov semigroup. Thus, a Poincar\'{e} inequality implies convergence of $P_t$ to the invariant measure (or equilibrium) $\nu$ with respect to the $|\cdot|_{2,\nu}$ norm and the Poincar\'{e} constant describes the rate of convergence.

Similarly a (tight) logarithmic Sobolev inequality is said to hold if there is a constant $L$ such that
\begin{equation*}
\mathrm{Ent}_\nu(f^2):=\int f^2\ln(f^2) d\nu - \int f^2 d\nu \ln\lp \int f^2 d\nu \rp \le L \tilde{\eE}(f)
\end{equation*}
for all $f\in\dD(\tilde{\eE})$ or for all $f$ in a subset that is dense in $\dD(\tilde{\eE})$ wrt.\ the norm defined by $\lp\tilde{\eE}(\cdot)+ |\cdot|_{2,\nu}^2\rp^{1/2}$.
By applying the logarithmic Sobolev inequality to $f:=1+\eps g$ for $g\in\dD(\tilde{\eE})$ with $\int g d\nu=0$ and letting $\eps$ tend to zero a logarithmic Sobolev inequality with constant $L$ can be seen to imply a Poincar\'{e} inequality with constant $L/2$. The logarithmic Sobolev inequality is thus the stronger notion. The logarithmic Sobolev inequality is equivalent to the following exponential decay in entropy property
\begin{equation*}
\mathrm{Ent}_\nu (P_t f) \le e^{-4t/L} \mathrm{Ent}_\nu (f)\ \forall 0<f\in L^1(\nu), t\ge0.
\end{equation*}
By Pinsker's inequality this implies convergence to the invariant measure in the total variation norm.\\

The central object of interest in this article will be Poincar\'{e} ($C$) and logarithmic Sobolev ($L$) constants in several different settings denoted specifically by
$C_\mu,\ L_\mu,\ \hat{C}_\mu,\ \hat{L}_\mu,\ C_{\laa},\ L_{\laa},\ C_{\sib},\ L_{\sib}$. In order of mention these are associated to the following processes (and infinitesimal generators): Brownian motion with sticky-reflecting boundary diffusion doubly weighted by $\mu$ on $\Om$, Brownian motion with sticky-reflection doubly weighted by $\mu$ on $\Om$, reflected Brownian motion weighted by $\al$ on $\Om$ ($\al$-weighted Laplacian on $\Om$ with Neumann boundary condition), and Brownian motion weighted by $\be$ on $\dOm$ ($\be$-weighted Laplacian on $\dOm$). These eight constants are respectively defined as the smallest constants fulfilling for all $f\in C^1(\Om)$
{\small
\begin{align*}
\var_\mu(f) &\le C_\mu \lp \int_\Om |\nabla f|^2 \al d\la + \int_{\dOm} |\nabla^\tau f|^2 \be d \si \rp, \ &\ent_\mu (f^2) &\le L_\mu \lp \int_\Om |\nabla f|^2 \al d\la + \int_{\dOm} |\nabla^\tau f|^2 \be d \si\rp, \\
\var_\mu(f) &\le \hat{C}_\mu \int_\Om |\nabla f|^2 \al d\la, \ &\ent_\mu (f^2) &\le \hat{L}_\mu \int_\Om |\nabla f|^2 \al d\la, \\
\var_{\laa}(f) &\le C_{\laa} \int_{\Om} |\nabla f|^2 d\laa, \ &\ent_{\laa}(f^2) &\le L_{\laa} \int_{\Om} |\nabla f|^2 d\laa, 
\end{align*}
}
and for all $f\in C^1(\dOm)$
\begin{equation*}
\var_{\sib}(f) \le C_{\sib} \int_{\dOm} |\nabla^\tau f|^2 d\sib, \quad \ent_{\sib}(f^2) \le L_{\sib} \int_{\dOm} |\nabla^\tau f|^2 d\sib.
\end{equation*}

Functional inequalities associated to simple reflected weighted Brownian motion on $\Om$ or weighted Brownian motion on $\dOm$, i.e.\ more standard boundary behaviour, are more well-understood. References providing upper bounds on these Poincar\'{e} and logarithmic Sobolev constants include respectively~\cite{MR889476,MR1262968,MR1450586,MR2158015} and~\cite{MR889476,MR1307413,MR1432586,MR1452551,MR1698947,MR1064685}. For the doubly weighted case with sticky reflection and with or without boundary diffusion it will be part of the results in this paper that Poincar\'{e} and logarithmic Sobolev inequalities hold at all. Additionally we give upper bounds on the associated constants $C_\mu,\ L_\mu,\ \hat{C}_\mu,\ \hat{L}_\mu$ in terms of $C_{\laa},L_{\laa}, C_{\sib},L_{\sib}$. \\
The study of the Poincar\'{e} and logarithmic Sobolev constants $C_\mu,L_\mu,\hat{C}_\mu, \hat{L}_\mu$ will also turn out to be connected to doubly weighted Steklov eigenvalues:
We write $\sigma_{\mu,1}$ for the first positive eigenvalue in the doubly weighted Steklov eigenvalue problem
\begin{equation*}
	\begin{cases}
		\Delta_\al f = 0 &\text{in } \Om^\circ,\\
		\al \frac{\partial f}{\partial N} = \sigma_{\mu} \be f&\text{on } \dOm .
	\end{cases}
\end{equation*}
Furthermore, we write $f_\al$ for the harmonic (wrt.\ $\Delta_\al$) extension to $\Om$ of a function $f$ on $\dOm$ and $D_\al$ for the Dirichlet-to-Neumann operator $D_\al:C^\infty(\dOm)\to C^\infty(\dOm), D_\al f:= \frac{\partial f_\al}{\partial N}$.
By definition the number $\sigma_{\mu,1}$ is then the first eigenvalue of the doubly weighted Dirichlet-to-Neumann operator $\frac{\al}{\be} D_\al$. If $\al=const \in \lp 0, \frac{1}{|\Om|}\rp$ only the boundary condition is weighted. See~\cite{colbgir} for more details on this singly weighted Steklov eigenvalue.

\subsection{Assumptions}

We now state and explain the assumptions on the two weight functions $\al,\be:\Om\to\mathbb{R}$ underlying the rest of the article:\\
\begin{assumption}
\begin{enumerate}[label*=\arabic*.]
\item Let the weights $\al,\be$ be such that $\mu$ is a probability measure. \label{as:mu}
\item Let $\al,\be \in C(\Om), \al>0\ \la$-a.e on $\Om$ and $\be>0\ \si$-a.e. on $\dOm$ such that $\sqrt{\al}\in W^{1,2}(\Om)$ and $\sqrt{\be}\in W^{1,2}(\dOm)$. \label{as:grothaus}
\item Let the weighted Poincar\'{e} and logarithmic Sobolev constants $C_{\laa},L_{\laa},C_{\sib},L_{\sib}$ be finite and their exact values or upper bounds on them be known.  \label{as:standardconstants}
\item Let $\al,\be$ be such that the following Sobolev embeddings and traces hold:
\begin{enumerate}[label*=\arabic*.]
\item $W^{1,2}(\Om,\al d\la)\to L^2(\Om,\al d\la)$ compactly \label{as:embeddings1}
\item $W^{1,2}(\Om,\al d\la)\to L^p(\Om,\be d\la)$ for $p\in\left[2,\frac{2d-2}{d-2}\right]$  \label{as:embeddings2}
\item $W^{1,2}(\dOm,\be d\si)\to L^2(\dOm,\be d\si)$ compactly  \label{as:embeddings3}
\item $W^{1,2}(\Om,\al d\la)\to L^2(\dOm,\be d\si)$ compactly \label{as:embeddings4}
\end{enumerate} 
\end{enumerate}
\end{assumption}

In the above~\ref{as:mu} is assumed as it is customary to formulate Poincar\'{e} and logarithmic Sobolev inequalities with respect to probability measures. Assumption~\ref{as:mu} implies that $A,B\in(0,1)$ and $A+B=1$. 
The conditions written in~\ref{as:grothaus} are simple and sufficient conditions in~\cite{gv}, for which we rely on the construction of doubly weighted Brownian motion with sticky reflection and with or without boundary diffusion. Note that the condition $\sqrt{\be}\in W^{1,2}(\dOm)$ is not used in the case $\delta=0$ without boundary diffusion. We refer to~\cite{gv} for possible weaker conditions than those written in~\ref{as:grothaus} and phrase the following assumptions independent of assumption 2. In~\ref{as:standardconstants} the finiteness condition suffices in order to establish that Poincar\'{e} and logarithmic Sobolev inequalities hold for doubly weighted Brownian motion with sticky reflection and with or without boundary diffusion, while the exact values/upper bounds are needed in order to obtain explicit upper bounds on $C_\mu,L_\mu,\hat{C}_\mu,\hat{L}_\mu$ (see Propositions~\ref{prop:varinterpol},~\ref{prop:pcsrbm},~\ref{prop:logsobinterpol},~\ref{prop:logsobinterpolwobd}).
Assumptions~\ref{as:embeddings1} and~\ref{as:embeddings2} are sufficient in order to ensure the existence of the doubly weighted Sobolev-Poincar\'{e} constants $C^{\al,\be}_{p,2}$, where the latter will quantitatively enter the upper bounds on $L_\mu,\hat{L}_\mu$. Furthermore~\ref{as:embeddings1},~\ref{as:embeddings3} and~\ref{as:embeddings4} will be used to establish that Poincar\'{e} inequalities hold at all for doubly weighted Brownian motion with sticky reflection and with or without boundary diffusion (see Propositions~\ref{prop:wexpc} and~\ref{prop:expcwobd}) but do not enter the results quantitatively. Moreover~\ref{as:embeddings4} will be used to show feasibility of the approach used below in order to give quantitative upper bounds on $C_\mu,L_\mu,\hat{C}_\mu,\hat{L}_\mu$ (see Remark~\ref{rem:exk1k2k}), in order to actually show upper bounds on them in the case of equal weights $\al=\be$ (see Theorem~\ref{thm:equalweights}) as well as when establishing the connection with the first nontrivial doubly weighted Steklov eigenvalue (see Lemma~\ref{lem:steklovvarchar}), but again the Sobolev trace operator does not enter the quantitative results. 

\section{Poincar\'{e} Inequality}\label{sec:PI}

In this section we generally consider the case $\delta=1$ including weighted boundary diffusion and only discuss the case $\delta=0$ without boundary diffusion in subsection~\ref{ssec:pcwobd}.
Here we will show that a Poincar\'{e} inequality holds for Brownian motion with sticky-reflecting boundary diffusion doubly weighted by $\mu$ on $\Om$ and give upper bounds on the Poincar\'{e} constant $C_\mu$. The bounds will depend on the Poincar\'{e} constants $C_{\laa}$ and $C_{\sib}$.

\begin{remark}
Note that for the definition of the diffusion process as well as the functional inequalities the function $\be$ is only relevant on $\dOm$. However, similarly to~\cite{gv} the results presented in the following are formulated in terms of a function $\be$ defined on the whole $\Om$. Thus in order to improve the results we could optimise over extensions of $\be\vert_{\dOm}$, but refrain from implementing this here.
\end{remark}

\begin{proposition}\label{prop:wexpc}
There is a constant $C_\mu\in\mathbb{R}$ such that for all $f\in C^1(\Om)$
\begin{equation*}
\mathrm{Var}_{\mu}(f) \le C_\mu \eE_\mu(f).
\end{equation*}
\end{proposition}
\begin{proof}
Assume the claim is false, i.e.\ there is a sequence $(f_k)_{k\in\mathbb{N}}\subset C^1(\Om)$ such that 
\begin{equation}\label{eq:PIexistencew}
\mathrm{Var}_{\mu}(f_k)>k \eE_\mu(f_k).
\end{equation} Assume without loss of generality that $\int_\Om f_k d\mu =0$ and $\mathrm{Var}_{\mu}(f_k)=1$ for all $k\in\mathbb{N}$. Then inequality~\eqref{eq:PIexistencew} implies that
\begin{equation*}
\eE_\mu(f_k)<\frac{1}{k}\ \Leftrightarrow\ \int_\Om |\nabla f_k|^2 \al d\la + \int_{\dOm} |\nabla^\tau f_k|^2 \be d\si <\frac{1}{k}.
\end{equation*}
Furthermore, 
\begin{equation*}
1= \mathrm{Var}_{\mu}(f_k) = A \mathrm{Var}_{\laa}(f_k) + B\mathrm{Var}_{\sib}(f_k) +AB\lp\int_\Om f_k d\laa - \int_{\dOm} f_k d\sib\rp^2 
\end{equation*}
together with $\int_\Om f_k d\mu =0$ implies boundedness of the sequence in $|\cdot|_{2,\laa}$ and $|\cdot|_{2,\sib}$:
\begin{align*}
|f_k|^2_{2,\laa} &= \mathrm{Var}_{\laa}(f_k) + \lp\int_\Om f_k d\laa\rp^2=  \mathrm{Var}_{\laa}(f_k) + \lp B\int_\Om f_k d\laa+A\int_\Om f_k d\laa\rp^2\\
	&= \mathrm{Var}_{\laa}(f_k) + \lp B\int_\Om f_k d\laa- B\int_{\dOm} f_k d\sib\rp^2 \\
	&= \mathrm{Var}_{\laa}(f_k) + B^2 \lp\int_\Om f_k d\laa- \int_{\dOm} f_k d\sib\rp^2 \le \frac{1}{A} + \frac{B}{A},\\
|f_k|^2_{2,\sib}&= \mathrm{Var}_{\sib}(f_k) + \lp \int_{\dOm} f_k d\sib \rp^2 \\
	&=  \mathrm{Var}_{\sib}(f_k) + \lp A\int_{\dOm} f_k d\sib +B\int_{\dOm} f_k d\sib\rp^2\\
	&= \mathrm{Var}_{\sib}(f_k) + \lp A\int_{\dOm} f_kd\sib - A\int_{\Om} f_k d\laa \rp^2 \\
	&= \mathrm{Var}_{\sib}(f_k) + A^2\lp \int_{\dOm} f_kd\sib - \int_{\Om} f_k d\laa \rp^2 \le \frac{1}{B} + \frac{A}{B}.
\end{align*}
Thus, $(f_k)_k$ is a bounded sequence in $W^{1,2}(\Om,\al d\la)$ and (by restriction to $\dOm$) in $W^{1,2}(\dOm,\be d\si)$. The following is up to taking subsequences: As $W^{1,2}(\Om,\al d\la)$ and $W^{1,2}(\dOm,\be d\si)$ are Hilbert spaces and thus reflexive $(f_k)_k$ converges weakly to an element $f$ in $W^{1,2}(\Om,\al d\la)$ and (by restriction to $\dOm$) in $W^{1,2}(\dOm,\be d\si)$. As $W^{1,2}(\Om,\al d\la)$ is compactly embedded in $L^2(\Om,\al d\la)$ and $W^{1,2}(\dOm,\be d\si)$ is compactly embedded in $L^2(\dOm,\be d\si)$, $(f_k)_k$ converges to $f$ in $L^2(\Om,\al d\la)$ and in $L^2(\dOm,\be d\si)$. Thus, 
\begin{equation*}
1= \lim_{k\to\infty} \mathrm{Var}_{\mu}(f_k) = \lim_{k\to\infty} |f_k|^2_{2,\mu} = |f|^2_{2,\mu}.
\end{equation*} 
Furthermore, $\int_\Om f d\mu=0$. Since $f\in W^{1,2}(\Om,\al d\la)$ with $\nabla f =0$, $f$ has to be constant $\la$ almost surely. If $f\equiv c$ almost surely for $c\neq0$, then also $f\vert_{\dOm}\equiv c$ and $\int_\Om f d\mu=c\neq 0$. Thus, $f\equiv 0$ which is a contradiction to $|f|_{2,\mu}=1$.
\end{proof}

The basis of the explicit upper bounds on $C_\mu$ will be the following Proposition, which is a doubly weighted analogue of a result in~\cite{mvr} and bounds $C_\mu$ using in particular $C_{\laa}$ and $C_{\sib}$.

\begin{proposition}\label{prop:varinterpol}
Assume there are constants $K_1,K_2,K_{\dOm,\Om}$ such that for all $f\in C^1(\Om)$
\begin{align}
	\var_{\sib}(f) &\le K_{\dOm,\Om} \int_\Om |\nabla f|^2 d\laa \label{eq:varboundary},\\
	\lp \int_\Om f d\laa - \int_{\dOm} f d\sib \rp^2 &\le K_1 \int_\Om |\nabla f|^2 d\laa + K_2\int_{\dOm} |\nabla^\tau f|^2 d\sib, \label{eq:mixedterm}
\end{align}
	then
	\begin{equation*}
		C_\mu \le \max \lp C_{\laa} + BK_1, A K_2, \frac{\frac{B}{A}K_{\dOm,\Om}( C_{\sib} + A K_2)+ (C_{\laa} + BK_1)C_{\sib}}{\frac{B}{A}K_{\dOm,\Om} + C_{\sib}}\rp.
	\end{equation*}
\end{proposition}
\begin{proof}
\begin{align*}
\var_\mu(f) &= \int_\Om f^2 \al d\la + \int_{\dOm} f^2 \be d \si - \lp \int_\Om f \al d \la + \int_{\dOm} f \be d\si \rp^2\\
	&= A\var_{\laa}(f) + B\var_{\sib}(f) +A\lp \int_\Om f d\laa \rp^2 + B\lp \int_{\dOm} f d\sib \rp^2\\
	& - \lp \int_\Om f \al d \la + \int_{\dOm} f \be d\si \rp^2\\
	&= A\var_{\laa}(f) + B\var_{\sib}(f) + \lp \lp\frac{1-A}{A}\rp^{1/2}\int_\Om f \al d\la - \lp\frac{1-B}{B}\rp^{1/2}\int_{\dOm} f \be d\si \rp^2 \\
	&=  A\var_{\laa}(f) + B\var_{\sib}(f) + AB\lp \int_\Om f d\laa - \int_{\dOm} f d\sib \rp^2 \\
	&\le \lp C_{\laa} + t \frac{B}{A}K_{\dOm,\Om} + BK_1 \rp \int_\Om |\nabla f|^2 \al d\la + \lp (1-t)C_{\sib} + A K_2\rp \int_{\dOm} |\nabla^\tau f|^2 \be d\si.
\end{align*}
Thus
\begin{align*}
	C_\mu &\le \inf_{t\in[0,1]} \max \lp C_{\laa} + t \frac{B}{A}K_{\dOm,\Om} + BK_1 , (1-t)C_{\sib} + A K_2 \rp\\
	&= \inf_{t\in[0,1]} \max \lp C_{\laa} + BK_1 + t \frac{B}{A}K_{\dOm,\Om} , C_{\sib} + AK_2 - tC_{\sib} \rp.
\end{align*}
Using that for $a,b,c,d>0$
\begin{equation*}
	\inf_{t\in[0,1]} \max(a+bt,c-dt) = \begin{cases} a &\text{if } c-a<0,\\
		c-d &\text{if } c-a>b+d,\\ 
		\frac{bc+ad}{b+d} &\text{if } 0\le c-a \le b+d,\end{cases}
\end{equation*}
it follows that 
\begin{equation*}
	C_\mu \le \begin{cases} C_{\laa} + BK_1 &\text{if } C_{\sib} + AK_2 - \lp C_{\laa} + BK_1  \rp <0, \\
		A K_2 &\text{if } C_{\sib} + A K_2 - \lp C_{\laa} + BK_1  \rp > \frac{B}{A}K_{\dOm,\Om} +C_{\sib}, \\
		\frac{\frac{B}{A}K_{\dOm,\Om}( C_{\sib} + A K_2)+ (C_{\laa} + BK_1)C_{\sib}}{\frac{B}{A}K_{\dOm,\Om} + C_{\sib}} &\text{if } 0\le C_{\sib} + A 	K_2 - \lp C_{\laa} + BK_1  \rp \le \frac{B}{A}K_{\dOm,\Om} +C_{\sib}. \end{cases}
\end{equation*}
\end{proof}

In the above in order to bound $\var_{\sib}(f)$ we interpolate between the two possible upper bounds given via inequality~\eqref{eq:varboundary} and via the Poincar\'{e} inequality for the $\beta$-weighted Laplacian on $\dOm$. 
While using inequality~\eqref{eq:varboundary} is not necessary to obtain an upper bound on $C_\mu$ at all in the case $\delta=1$, it will turn out to improve the upper bound.

\begin{remark}\label{rem:exk1k2k}
Note that the existence of constants $K_{\dOm,\Om}, K_1, K_2$ fulfilling inequalities~\eqref{eq:varboundary} and~\eqref{eq:mixedterm} follows by continuity of the Sobolev trace operator $\vert_{\dOm}$ and the Poincar\'{e} inequality for the $\al$-weighted Neumann Laplacian on $\Om$. More precisely, denoting the operator norm of $\vert_{\dOm}$ by $K$ and assuming without loss of generality that $\int_\Om f d\laa=0$ in inequalities~\eqref{eq:varboundary} and~\eqref{eq:mixedterm} it holds
{\small
\begin{align*}
\mathrm{Var}_{\sib} (f) \le \int_{\dOm} f^2 d\sib &\le \frac{A K^2}{B}\frac{|\Om|}{|\dOm|} \lp \int_\Om |\nabla f|^2 d\laa + \int_\Om f^2 d\laa \rp  \le \frac{A K^2}{B}\frac{|\Om|}{|\dOm|}(1+C_{\laa}) \int_\Om |\nabla f|^2 d\laa,\\
\lp \int_{\dOm} f d\sib \rp^2 \le \int_{\dOm} f^2 d\sib &\le \frac{A K^2}{B}\frac{|\Om|}{|\dOm|} \lp \int_\Om |\nabla f|^2 d\laa + \int_\Om f^2 d\laa \rp \le  \frac{A K^2}{B}\frac{|\Om|}{|\dOm|}(1+C_{\laa}) \int_\Om |\nabla f|^2 d\laa,
\end{align*}
}
i.e.\ inequalities~\eqref{eq:varboundary} and ~\eqref{eq:mixedterm} hold with 
\begin{equation*}
K_{\dOm,\Om}=\frac{A K^2}{B}\frac{|\Om|}{|\dOm|}(1+C_{\laa}) ,\ K_1=\frac{AK^2}{B}\frac{|\Om|}{|\dOm|}(1+C_{\laa}), \ K_2=0.
\end{equation*}
Optimal values in Sobolev trace inequalities have been studied e.g.\ in~\cite{MR1443055, MR1978428, MR1971310,MR2258478, MR2384747}. However, the exact value of $K$ or explicit bounds in terms of the geometry of the manifold seem to not be known in general. Therefore, the approach presented in the following does not make use of the Sobolev trace operator quantitatively. Instead the majority of this section will consist of finding explicit values for $K_{\dOm,\Om}, K_1, K_2$ without using the trace operator, and in turn we will obtain a bound on $K$ as a byproduct later, see Remark~\ref{rem:Sobtracebound} and Proposition~\ref{prop:sobtracebound}.
\end{remark}

The optimal constant $K_{\dOm,\Om}$ fulfilling \eqref{eq:varboundary} is strongly related to the first nontrivial doubly weighted Steklov eigenvalue $\sigma_{\mu,1}$. The variational characterisation of the first non-trivial eigenvalue $\sigma_{\mu,1}$ is given by the following lemma:

\begin{lemma}\label{lem:steklovvarchar}
\begin{equation*}
	\sigma_{\mu,1} = \frac{A}{B} \inf_{\substack{0\neq f\in C^1(\Om)\\ \int_{\dOm} f d\sib=0}}\frac{\int_\Om |\nabla f|^2 d\laa}{\int_{\dOm} f^2 d\sib}.
\end{equation*}
\end{lemma}
\begin{proof}
We first show existence of a minimiser for the right-hand side in $W^{1,2}(\Om,\al d\la)$. Denote by $\eta$ the right-hand side of the statement and let $(f_n)_n$ be a sequence in $C^1(\Om)$ such that $\int_{\dOm} f_n d\sib=0$ $\forall n\in\mathbb{N}$ and
\begin{equation*}
\lim_{n\to\infty} \frac{A}{B}\frac{\int_\Om |\nabla f_n|^2 d\laa}{\int_{\dOm} f_n^2 d\sib} = \eta.
\end{equation*}
Assume without loss of generality that $\int_\Om f_n^2 d\mu=1$ $\forall n\in\mathbb{N}$. 
Then for $\eps>0$ and $n$ sufficiently large
\begin{align*}
\int_\Om |\nabla f_n|^2 \al d\la \le \lp\eta + \eps\rp \int_{\dOm} f_n^2 \be d\si \le \eta + \eps \text{ and } \int_\Om f_n^2 \al d\la \le 1.
\end{align*}
Thus, $(f_n)_n$ is a bounded sequence in $W^{1,2}(\Om,\al d\la)$. The following arguments are up to taking subsequences: As a bounded sequence in the Hilbert space $W^{1,2}(\Om,\al d\la)$ $(f_n)_n$ converges weakly to an element $f\in W^{1,2}(\Om,\al d\la)$. Due to compactness of the Sobolev embedding $W^{1,2}(\Om,\al d\la)\to L^2(\Om,\al d\la)$ $(f_n)_n$ converges also in $L^2(\Om,\al d\la)$ and due to compactness of the trace operator $W^{1,2}(\Om,\al d\la)\to L^2(\dOm,\be d\si)$ the traces also converge in $L^2(\dOm,\be d\si)$. 
Thus, $\int_{\dOm} f d\sib=0$ and the value $\eta$ is attained for $f$ as
\begin{equation*}
\eta \le \frac{A}{B}\frac{\int_\Om |\nabla f|^2 d\laa}{\int_{\dOm} f^2 d\sib} \le \lim_{n\to\infty} \frac{A}{B}\frac{\int_\Om |\nabla f_n|^2 d\laa}{\int_{\dOm} f_n^2 d\sib} = \eta.
\end{equation*}
The latter is due to $C^1(\Om)$ functions being dense in $W^{1,2}(\Om,\al d\la)$ and due to lower semicontinuity of $|\nabla \cdot|_{2, \al d\la}$ for weak convergence in $W^{1,2}(\Om,\al d\la)$.\\
Now let $f\in W^{1,2}(\Om,\al d\la)$ with $\int_{\dOm} f d\sib=0$ be a minimiser for $\eta$, i.e.\
\begin{equation}\label{eq:fmin}
\eta= \frac{A}{B} \frac{\int_\Om |\nabla f|^2d\laa}{\int_{\dOm} f^2 d\sib}.
\end{equation} We consider the variation in direction $\xi$
\begin{equation*}
\frac{d}{d\eps} \frac{\int_{\Om} |\nabla(f+\eps\xi)|^2 d\laa}{\int_{\dOm} (f+\eps\xi)^2 d\sib}\vert_{\eps=0} = \frac{ 2\int_\Om \nabla\xi\cdot\nabla fd\laa\int_{\dOm} f^2 d\sib - 2\int_\Om |\nabla f|^2 d\laa \int_{\dOm} f\xi d\sib }{\lp \int_{\dOm} f^2 d\sib\rp^2} \stackrel{!}{=} 0.
\end{equation*} 
This implies by inserting equation~\eqref{eq:fmin} that
\begin{equation}\label{eq:fmin2}
\int_{\dOm} f^2 d\sib \left[ \int_\Om \nabla f\cdot\nabla\xi d\laa - \eta \frac{B}{A}\int_{\dOm} f \xi d\sib \right] = 0,
\end{equation}
and thus in case $\xi\vert_{\dOm}=0$
\begin{equation*}
\int_\Om \nabla f \cdot \nabla \xi d\laa = 0\ \Rightarrow \ \Delta_\al f = 0 \text{ in } \Om^\circ.
\end{equation*}
Furthermore, since
\begin{equation*}
\int_{\Om} \nabla f \cdot \nabla \xi \frac{\al}{A} d\la = - \int_\Om \xi \Delta_\al f \frac{\al}{A} d\la + \int_{\dOm} \xi \frac{\partial f}{\partial N} \frac{\al}{A} d\si = \int_{\dOm} \xi \frac{\partial f}{\partial N} \frac{\al}{A} d\si 
\end{equation*}
we have by equation~\eqref{eq:fmin2}
\begin{equation*}
 \int_{\dOm} \xi \frac{\partial f}{\partial N} \frac{\al}{A} d\si = \frac{\eta}{A}\int_{\dOm} f \xi \be d\si
\end{equation*}
and thus 
\begin{equation*}
\frac{\partial f}{\partial N} \al = \eta f \be \text{ on } \dOm,
\end{equation*}
i.e.\ $f$ is an eigenfunction for the first non-trivial eigenvalue in the doubly-weighted Steklov problem with eigenvalue $\eta\ge\sigma_{\mu,1}$. On the other hand inserting an eigenfunction for $\sigma_{\mu,1}$ in the quotient on the right-hand side of the statement of the Lemma implies $\eta\le \sigma_{\mu,1}$.
\end{proof}

\begin{remark}\label{rem:steklov}
If we write $K_{\dOm,\Om}$ for the optimal constant fulfilling~\eqref{eq:varboundary}, we may now note that
\begin{equation}\label{eq:steklov}
	K_{\dOm,\Om} =  \sup_{f\in C^1(\Om)} \frac{\var_{\sib}(f)}{\int_\Om |\nabla f|^2d\laa} = \sup_{\substack{f\in C^1(\Om)\\ \int_{\dOm}f d\sib=0}} \frac{\int_{\dOm} f^2 d\sib}{\int_\Om |\nabla f|^2d\laa} = \frac{A}{B}(\sigma_{\mu,1})^{-1}. 
\end{equation}
Therefore finding upper bounds on $K_{\dOm,\Om}$ corresponds to finding lower bounds on the first doubly weighted Steklov eigenvalue. 
\end{remark}

In the remaining part of this section we proceed to find explicit constants $K_{\dOm,\Om},K_1,K_2$ fulfilling conditions \eqref{eq:varboundary} and \eqref{eq:mixedterm}. Note that we could additionally optimise over pairs $K_1, K_2$ fulfilling~\eqref{eq:mixedterm} in order to improve the upper bound on $C_\mu$, but refrain from doing so here.

\subsection{Coinciding weights, positive curvature and convex boundary}\label{ssec:coinweights}

We first consider the case where $\al\vert_{\dOm}=\be\vert_{\dOm}$ and assume positive $n$-weighted Ricci curvature and nonnegative second fundamental form on the boundary in order to proceed analogously to~\cite{mvr}. The main results providing upper bounds on $C_\mu$ in this setting will be via the combination of Proposition~\ref{prop:varinterpol} and Theorem~\ref{thm:equalweights} as well as Theorem~\ref{thm:directboundcoincweights}.
The more general setting will be considered in subsection~\ref{ssec:genweights}.

\begin{theorem}\label{thm:equalweights}
Assume that $\be=\al$ on $\dOm$, $\mathrm{Ric}_{\al,n}\ge k_{\al,n}g$ for some $n\in[d,\infty]$, $k_{\al,n}>0$, $\mathrm{I\!I} \ge k_2 id$ with $k_2\ge0$ and $\int_{\dOm} \mathrm{H}_\al \al d\si\ge0$ on $\dOm$. Then inequality~\eqref{eq:mixedterm} is fulfilled with $K_2=0$ and
\begin{equation*}
	K_1 = \frac{n-1}{nk_{\al,n}}.
\end{equation*}
\end{theorem}

\begin{proof}
\allowdisplaybreaks
For $f\in C^1(\Om)$ we aim to bound from below
\begin{equation*}
	\frac{\int_\Om |\nabla f|^2 d\laa}{\lp \int_\Om f d\laa - \int_{\dOm} f d\sib \rp^2}.
\end{equation*}
\begin{equation*}
	\inf_{f\in C^1(\Om)} \frac{\int_\Om |\nabla f|^2 d\laa}{\lp \int_\Om f d\laa - \int_{\dOm} f d\sib \rp^2} = \inf_{\stackrel{f\in C^1(\Om)}{ \int_{\dOm} f d\sib=0}} \frac{\int_\Om |\nabla f|^2 d\laa}{\lp \int_\Om f d\laa \rp^2} \ge  \inf_{\stackrel{f\in C^1(\Om)}{\int_{\dOm} f d\sib=0}} \frac{\int_\Om |\nabla f|^2 d\laa}{\int_\Om f^2 d\laa} =:\eta.
\end{equation*}
Let $(f_k)_k$ be a sequence in $C^1(\Om)$ such that $\int_{\dOm} f_k d\sib=0$ for all $k\in\mathbb{N}$ and such that
\begin{equation*}
\lim_{k\to\infty} \frac{\int_\Om |\nabla f_k|^2 d\laa}{\int_\Om f_k^2 d\laa} = \eta.
\end{equation*}
We may assume without loss of generality that $\int_\Om f_k^2 d\laa=1$. The following arguments are up to taking subsequences: As a bounded sequence in the Hilbert space $W^{1,2}(\Om,\al d\la)$ $f_k$ converges weakly to an element $f\in W^{1,2}(\Om,\al d\la)$. By compactness of the embedding $W^{1,2}(\Om,\al d\la)\to L^2(\Om,\al d\la)$ and of the trace operator $W^{1,2}(\Om,\al d\la)\to L^2(\dOm,\be d\si)$ $(f_k)_k$ converges also in $L^2(\Om,\al d\la)$ and the traces converge in $L^2(\dOm,\be d\si)$. Thus, $\int_{\dOm} f d\sib=0$ and the value $\eta$ is attained for $f$ as
\begin{equation*}
\eta \le \frac{\int_\Om |\nabla f|^2 d\laa}{\int_\Om f^2 d\laa} \le \lim_{k\to\infty} \frac{\int_\Om |\nabla f_k|^2 d\laa}{\int_\Om f_k^2 d\laa} = \eta.
\end{equation*}
The latter is due to $C^1(\Om)$ functions being dense in $W^{1,2}(\Om,\al d\la)$ and due to the lower semicontinuity of $|\nabla \cdot|_{2, \al d \la}$ for weak convergence in $W^{1,2}(\Om,\al d\la)$. Varying $f$ in the direction of $g$ results in
\begin{align*}
	\eta \int_\Om fg d\laa &= \int_\Om \nabla f\cdot \nabla g d\laa  \\
	\Leftrightarrow\ \eta \int_\Om fg d\laa &= -\int_\Om g \Delta f d\laa + \int_{\dOm} g \frac{\partial f}{\partial N} \frac{\al}{A} d\si - \int_\Om g \nabla f \cdot \frac{\nabla \al}{\al} d\laa 
\end{align*}
for every $g\in C^1(\Om)$ with $\int_{\dOm} g d\sib =0$. By choosing $g\in C^\infty_0(\Om)$ we may deduce that 
\begin{equation*}-\Delta f - \nabla f\cdot\frac{\nabla \al}{\al}=\eta f \text{ in } \Om^\circ.
\end{equation*} This implies that $\int_{\dOm} g \frac{\partial f}{\partial N}\al d\si =0$ for $g$ with zero $\sib$-mean, and thus
\begin{equation*}
	\int_{\dOm} \frac{\partial f}{\partial N} \lp g- \int_{\dOm} g d\sib \rp \al d\si =0\ \Leftrightarrow\ \int_{\dOm} \lp \frac{\al}{\be}\frac{\partial f}{\partial N} - \int_{\dOm} \frac{\al}{\be}\frac{\partial f}{\partial N} d\sib \rp g \be d\si =0
\end{equation*}
for every $g\in C^1(\Om)$. Thus $\frac{\al}{\be}\frac{\partial f}{\partial N}$ is constant on $\dOm$. Summed up $f$ satisfies
\begin{equation*}
	\begin{cases} \Delta f + \nabla f\cdot \frac{\nabla \al}{\al} = -\eta f & \text{ in } \Om^\circ,\\
		\frac{\al}{\be}\frac{\partial f}{\partial N} = const. & \text{ on } \dOm,\\
		\int_{\dOm} f d\sib =0.\end{cases}
\end{equation*} 

We write the generalised Reilly formula, cf.~\cite{milman}:

\begin{align}\label{eq:genrei}
	\int_\Om \lp \Delta f + \frac{\nabla \al }{\al}\cdot \nabla f \rp^2 \al d\la &= \int_\Om \|\nabla^2 f\|^2 \al d\la + \int_\Om \mathrm{Ric}_\al (\nabla f, \nabla f) \al d\la + \int_{\dOm} \mathrm{H}_\al \lp \partial_N f\rp^2 \al d\si\\ 
	&+ \int_{\dOm} \mathrm{I\!I}(\nabla^\tau f, \nabla^\tau f) \al d\si - 2\int_{\dOm} \nabla^\tau (\partial_N f)\cdot  \nabla^\tau f \al d\si \notag.
\end{align}
	
Furthermore, (cf.~\cite{milman}) for $n\in[d,\infty]$:
\begin{equation}\label{eq:hesric}
	\mathrm{Ric}_\al (\nabla f,\nabla f) + \|\nabla^2 f\|^2 \ge \mathrm{Ric}_{\al,n}(\nabla f,\nabla f) + \frac{1}{n}\lp \Delta f + \frac{\nabla \al}{\al}\cdot \nabla f \rp^2.
\end{equation}

For the term on the left-hand side of the generalised Reilly formula we may calculate as follows
\begin{align*}
	\int_\Om \lp \Delta f + \frac{\nabla \al }{\al}\cdot \nabla f \rp^2 \al d\la &= -\eta \int_\Om f \lp \Delta f + \frac{\nabla \al }{\al}\cdot \nabla f \rp \al d\la = \eta \int_\Om \nabla f\cdot\nabla f \al d\la - \eta\int_{\dOm} f \frac{\partial f}{\partial N} \al d\si \\
	&= \eta \int_\Om |\nabla f|^2 \al d\la -c\eta\int f\be d\si = \eta \int_\Om |\nabla f|^2 \al d\la.
\end{align*}
	
Concerning the right-hand side we have (here $\al=\be$ on $\dOm$ is used for the second and more importantly the fourth term)
\begin{align*}
	\int_\Om \|\nabla^2 f\|^2 \al d\la + \int_\Om \mathrm{Ric}_\al (\nabla f, \nabla f) \al d\la &\ge \int_\Om \mathrm{Ric}_{\al,n}(\nabla f,\nabla f) \al d\la + \int_\Om \frac{1}{n}\lp\Delta f + \frac{\nabla \al}{\al}\cdot \nabla f\rp^2 \al d\la \\
	&\ge k_{\al,n}\int_\Om |\nabla f|^2 \al d\la + \frac{\eta}{n} \int_\Om |\nabla f|^2 \al d\la,\\
	\int_{\dOm} \mathrm{H}_\al \lp \partial_N f\rp^2 \al d\si &= \int_{\dOm} \lp \mathrm{H} + \frac{\partial \al}{\partial N}\frac{1}{\al}\rp \lp \partial_N f\rp^2 \al d\si \ge 0,\\
	\int_{\dOm} \mathrm{I\!I}(\nabla^\tau f, \nabla^\tau f) \al d\si &\ge 0,\\
	\int_{\dOm} \nabla^\tau (\partial_N f)\cdot \nabla^\tau f \al d\si &= 0.
\end{align*}
Thus 
\begin{equation*}
	\frac{n-1}{n}  \eta \int_\Om |\nabla f|^2 \al d\la \ge k_{\al,n}\int_\Om |\nabla f|^2 \al d\la \ \Leftrightarrow \ \eta \ge \frac{n}{n-1} k_{\al,n}.
\end{equation*}
\end{proof}

Furthermore, due to~\eqref{eq:steklov} we may obtain an upper bound on $K_{\dOm,\Om}$ using existing lower bounds on the first nontrivial weighted Steklov eigenvalue, see e.g.~\cite[Theorem 1.5]{batista},~\cite[Theorem 1.10]{MR4723999}. Thus via Proposition~\ref{prop:varinterpol} an upper bound on $C_\mu$ follows.\\

Alternatively a direct approach to bound $C_\mu$ from above is presented in the next Theorem:

\begin{theorem}\label{thm:directboundcoincweights}
Assume that $\be=\al$ on $\dOm$, $\mathrm{Ric}_{\al,n}\ge k_{\al,n}g$ for some $n\in[d,\infty]$, $k_{\al,n}>0$, $\mathrm{I\!I}\ge k_2 id$ with $k_2>0$ and $H_\al\ge0$ on $\dOm$. Then it holds
\begin{equation*}
	C_\mu \le \max\lp \frac{3n-1}{n k_2} , \frac{n-1}{n k_{\al,n}} \rp.
\end{equation*}
\end{theorem}

\begin{proof}
\allowdisplaybreaks
Let $\eta$ be the first nontrivial eigenvalue for the infinitesimal generator $L$ and $f$ an associated eigenfunction, i.e.\
\begin{equation*}
	\begin{cases}
		\Delta f + \frac{\nabla \al}{\al}\cdot \nabla f = -\eta f &\text{in } \Om^\circ, \\
		\Delta^\tau f + \frac{\nabla^\tau \be}{\be}\cdot \nabla^\tau f -\frac{\al}{\be}\partial_N f = -\eta f &\text{on } \dOm.
	\end{cases}
\end{equation*}
Now applying the generalised Reilly formula~\eqref{eq:genrei} and additionally~\eqref{eq:hesric} to the function $f$ gives
\begin{align*}
	\frac{n-1}{n} \int_\Om \lp \Delta f + \frac{\nabla \al}{\al}\cdot \nabla f \rp^2 \al d\la &\ge k_{\al,n} \int_\Om |\nabla f|^2 \al d\la + \int_{\dOm} H_\al \lp \partial_N f\rp^2 \al d\si \\
	&+ \int_{\dOm} \mathrm{I\!I} (\nabla^\tau f, \nabla^\tau f) \al d\si - 2\int_{\dOm} \nabla^\tau (\partial_N f)\cdot  \nabla^\tau f \al d\si.
\end{align*}
For the left-hand side we calculate
\begin{align*}
	\int_\Om \lp \Delta f + \frac{\nabla \al}{\al}\cdot \nabla f \rp^2 \al d\la &= -\eta \int_\Om f\lp \Delta f + \frac{\nabla \al}{\al}\cdot \nabla f \rp \al d\la \\
	&= \eta \int_\Om |\nabla f|^2 \al d\la - \eta \int_{\dOm} f \frac{\al}{\be} \partial_N f \be d\si \\
	&= \eta \int_\Om |\nabla f|^2 \al d\la - \eta \int_{\dOm} f (\Delta^\tau f + \frac{\nabla^\tau \be}{\be}\cdot \nabla^\tau f  +\eta f) \be d\si \\
	&= \eta \int_\Om |\nabla f|^2 \al d\la + \eta \int_{\dOm} |\nabla^\tau f|^2 \be d\si - \eta^2 \int_{\dOm} f^2\be d\si \\
	&\le  \eta \int_\Om |\nabla f|^2 \al d\la + \eta \int_{\dOm} |\nabla^\tau f|^2 \be d\si.
\end{align*} 
Concerning the right-hand side (here we use $\al=\be$ on $\dOm$ for the third term)
\begin{align*}
	\int_{\dOm} \mathrm{H}_\al \lp \partial_N f\rp^2 \al d\si  &= \int_{\dOm} \lp\mathrm{H} + \frac{1}{\al}\partial_N \al\rp\cdot (\partial_N f)^2 \al d\si \ge 0, \\
	\int_{\dOm} \mathrm{I\!I} (\nabla^\tau f,\nabla^\tau f)\al d\si &\ge k_2 \int_{\dOm} |\nabla^\tau f|^2 \al d\si, \\
	\int_{\dOm} \nabla^\tau (\partial_N f)\cdot  \nabla^\tau f \al d\si &= -\int_{\dOm} \partial_N f \lp \Delta^\tau f + \frac{\nabla^\tau \al}{\al} \cdot \nabla^\tau f\rp\al d\si \\
	&= -\int_{\dOm} \lp \Delta^\tau f + \frac{\nabla^\tau \be}{\be}\cdot \nabla^\tau f  +\eta f \rp  \lp \Delta^\tau f + \frac{\nabla^\tau \al}{\al} \cdot \nabla^\tau f\rp\be d\si \\
	&= -\int_{\dOm} \lp \Delta^\tau f + \frac{\nabla^\tau \be}{\be}\cdot \nabla^\tau f \rp^2 \be d\si + \eta \int_{\dOm} |\nabla^\tau f|^2 \be d\si \\
	&\le \eta \int_{\dOm} |\nabla^\tau f|^2 \be d\si .
\end{align*}
Combining these we get
\begin{equation}\label{eq:anotherineq}
	\lp \frac{n-1}{n} \eta - k_{\al,n}\rp \int_\Om |\nabla f|^2 \al d\la \ge \lp -\frac{n-1}{n}\eta -2\eta + k_2\rp  \int_{\dOm} |\nabla^\tau f|^2 \be d\si.
\end{equation}
Now either the first factor on the right-hand side of inequality~\eqref{eq:anotherineq} is nonpositive, i.e.\
\begin{equation*}
	\lp -\frac{n-1}{n}\eta -2\eta + k_2\rp  = -\frac{3n-1}{n}\eta + k_2 \le 0 \Leftrightarrow \eta \ge \frac{n}{3n-1}k_2,
\end{equation*}
or if the first factor on the right-hand side of inequality~\eqref{eq:anotherineq} is nonnegative, then the first factor on the left-hand side of inequality~\eqref{eq:anotherineq} also has to be nonnegative, i.e.\
\begin{equation*}
	\frac{n-1}{n}\eta-k_{\al,n} \ge 0\ \Leftrightarrow\ \eta \ge k_{\al,n} \frac{n}{n-1}.
\end{equation*}
Thus
\begin{equation*}
	\eta \ge \min\lp \frac{n}{3n-1}k_2 , k_{\al,n} \frac{n}{n-1} \rp.
\end{equation*}
\end{proof}

\subsection{General weights and bounded curvature}\label{ssec:genweights}

We now allow for more general weights which do not have to coincide on $\dOm$ and assume bounds of arbitrary sign on the Ricci and sectional curvature on $\Om$ as well as on the second fundamental form on $\dOm$. We again proceed to find constants $K_{\dOm,\Om},K_1,K_2$ fulfilling conditions \eqref{eq:varboundary} and \eqref{eq:mixedterm}. As remarked above we could optimise over pairs $K_1, K_2$ fulfilling~\eqref{eq:mixedterm} in order to improve the upper bound on $C_\mu$. However, the choice of $K_1$ and $K_2=0$ presented next will also turn out to be useful when considering the case without boundary diffusion in subsection~\ref{ssec:pcwobd}. The main result providing upper bounds on $C_\mu$ in this setting will be the combination of Proposition~\ref{prop:varinterpol} and Theorem~\ref{thm:explconst}.

\begin{proposition}\label{prop:mixedterm}
Let $\varphi\in C^1(\Om)$ such that $\frac{\partial \varphi}{\partial N}\vert_{\dOm}=\pm 1$ and $\nabla \varphi$ is Lipschitz continuous on $\Om$. Then inequality~\eqref{eq:mixedterm} is fulfilled with $K_2=0$ and 
\begin{equation*}
K_1 = \frac{A}{B^2} \left|\frac{\be}{\al}\right|_{\infty}\lp \inf_{\eps\in(0,\infty)} (1+\eps) C_{\laa} \left|\Delta \varphi + \frac{\nabla \be}{\be}\cdot \nabla \varphi \right|_{2,\be d\la}^2 + (1+\eps^{-1}) |\nabla \varphi|_{2,\be d\la}^2\rp.
\end{equation*}
\end{proposition}

\begin{proof}
\allowdisplaybreaks
Let $f\in C^1(\Om)$. Without loss of generality we may assume $\int_\Om f d\laa=0$. For $\eps\in(0,\infty)$ we may then compute as follows:
\begin{align*}
&\lp \int_\Om f d\laa - \int_{\dOm} f d\sib\rp^2 = \lp \int_{\dOm} f d\sib\rp^2 = \lp \int_{\dOm} f\frac{\partial \varphi}{\partial N} d\sib \rp^2 \\
		&= \frac{1}{B^2}\lp \int_\Om f \be \Delta \varphi d\la + \int_\Om \nabla \varphi \cdot (\be \nabla f + f \nabla \be) d\la\rp^2 \\
		&= \frac{1}{B^2}\lp \int_\Om \lp f\Delta \varphi +f \frac{\nabla \be}{\be}\cdot \nabla \varphi\rp \be d \la + \int_\Om \nabla f \cdot \nabla \varphi \be d\la \rp^2 \\
		&\le \frac{1}{B^2}\lp(1+\eps) \lp \int_\Om f \lp \Delta \varphi + \frac{\nabla \be}{\be}\cdot \nabla \varphi\rp \be d\la \rp^2 + (1+\eps^{-1}) \lp \int_\Om \nabla f\cdot \nabla \varphi \be d\la \rp^2\rp \\
		&\le\frac{1}{B^2}\lp (1+\eps)  \int_\Om \lp \Delta \varphi + \frac{\nabla \be}{\be}\cdot \nabla \varphi\rp^2 \be d\la \int_\Om f^2 \be d\la + (1+\eps^{-1}) \int_\Om |\nabla f|^2 \be d\la  \int_\Om |\nabla \varphi|^2 \be d\la\rp \\
		&\le\frac{1}{B^2}\biggl( (1+\eps) \left|\frac{\be}{\al}\right|_\infty \int_\Om \lp \Delta \varphi + \frac{\nabla \be}{\be}\cdot \nabla \varphi\rp^2 \be d\la \int_\Om f^2 \al d\la + \\
			&(1+\eps^{-1}) \left|\frac{\be}{\al}\right|_\infty \int_\Om |\nabla f|^2 \al d\la  \int_\Om |\nabla \varphi|^2 \be d\la\biggr) \\
		&\le \frac{1}{B^2}\biggl( (1+\eps) \left|\frac{\be}{\al}\right|_\infty \int_\Om \lp \Delta \varphi + \frac{\nabla \be}{\be}\cdot \nabla \varphi\rp^2 \be d\la\ C_{\laa} \int_\Om |\nabla f|^2 \al d\la+ \\
			&(1+\eps^{-1})  \left|\frac{\be}{\al}\right|_\infty  \int_\Om |\nabla f|^2 \al d\la \int_\Om |\nabla \varphi|^2 \be d\la \biggr) \\
		&\le \frac{A}{B^2}\left|\frac{\be}{\al}\right|_\infty\lp (1+\eps) C_{\laa} \int_\Om \lp \Delta \varphi + \frac{\nabla \be}{\be}\cdot \nabla \varphi\rp^2 \be d\la  + (1+\eps^{-1})   \int_\Om |\nabla \varphi|^2 \be d\la \rp \int_\Om |\nabla f|^2 d\laa.
\end{align*}
\end{proof}

In the foregoing Proposition~\ref{prop:mixedterm} Lipschitz continuity of $\nabla \varphi$ is assumed in order to ensure that $\Delta \varphi$ is well-defined almost everywhere via Rademacher's theorem. We later construct $\varphi$ such that $\nabla \varphi$ is in fact Lipschitz continuous, however in general the assumption could be weakened. The same is true in the following Proposition for $\rho$ instead of $\varphi$.

\begin{proposition}\label{prop:varboundary}
Let $\rho\in C^1(\Om)$ such that $\frac{\partial \rho}{\partial N}\vert_{\dOm}=- 1$ and $\nabla \rho$ is Lipschitz continuous on $\Om$. Then inequality~\eqref{eq:varboundary} is fulfilled with 
\begin{equation*}
K_{\dOm,\Om} = \frac{A}{B}\left|\frac{\be}{\al}\right|_{\infty}\lp C_{\laa}\left|\lp \Delta \rho +\frac{\nabla \be}{\be} \cdot \nabla \rho \rp^-\right|_\infty   + 2C_{\laa}^{(1/2)}|\nabla \rho|_\infty \rp.
\end{equation*} 
\end{proposition}

\begin{proof}
Let $f\in C^1(\Om)$. Without loss of generality we may assume that $\int_\Om f d\laa =0$.
\begin{align}
\var_{\sib}(f) &= \int_{\dOm} f^2 d \sib - \lp \int_{\dOm} f d\sib \rp^2 \le \int_{\dOm} f^2 d \sib = -\frac{1}{B} \int_{\dOm} f^2 \frac{\partial \rho}{\partial N} \be d\si \notag\\
	&= -\frac{1}{B}\int_\Om f^2 \Delta \rho \be d\la - \frac{1}{B}\int_\Om \nabla \rho \cdot (2f\nabla f \be + f^2 \nabla \be) d\la \notag\\
	&= -\frac{1}{B}\int_\Om f^2 \lp \Delta \rho +\frac{\nabla \be}{\be} \cdot \nabla \rho \rp \be d\la - \frac{2}{B}\int_\Om f\nabla f \cdot \nabla \rho \be d\la \notag\\
	&\le \frac{1}{B}\int_\Om f^2 \lp \Delta \rho +\frac{\nabla \be}{\be} \cdot \nabla \rho \rp^- \be d \la + \frac{2}{B}\int_\Om |f||\nabla f| |\nabla \rho| \be d\la \notag\\
	&\le \frac{1}{B}\left|\lp \Delta \rho +\frac{\nabla \be}{\be} \cdot \nabla \rho \rp^-\right|_\infty \left|\frac{\be}{\al}\right|_\infty \int_\Om f^2 \al d \la +  \frac{2}{B}|\nabla \rho|_\infty \left|\frac{\be}{\al}\right|_\infty\int_\Om |f||\nabla f| \al d\la \notag\\
	&\le \frac{1}{B}\left|\lp \Delta \rho +\frac{\nabla \be}{\be} \cdot \nabla \rho \rp^-\right|_\infty \left|\frac{\be}{\al}\right|_\infty \int_\Om f^2 \al d \la \notag\\
		&+ \frac{2}{B}|\nabla \rho|_\infty \left|\frac{\be}{\al}\right|_\infty \lp \int_\Om f^2 \al d\la \int_\Om |\nabla f|^2 \al d\la \rp^{1/2} \label{eq:someeq1}\\
	&\le \frac{C_{\laa}}{B}\left|\lp \Delta \rho +\frac{\nabla \be}{\be} \cdot \nabla \rho \rp^-\right|_\infty \left|\frac{\be}{\al}\right|_\infty \int_\Om |\nabla f|^2 \al d \la + \frac{2 C_{\laa}^{1/2}}{B}|\nabla \rho|_\infty \left|\frac{\be}{\al}\right|_\infty\int_\Om |\nabla f|^2 \al d\la \notag\\
	&=\frac{A}{B}\left|\frac{\be}{\al}\right|_\infty\lp C_{\laa}\left|\lp \Delta \rho +\frac{\nabla \be}{\be} \cdot \nabla \rho \rp^-\right|_\infty   + 2C_{\laa}^{(1/2)}|\nabla \rho|_\infty \rp \int_\Om |\nabla f|^2 d\laa. \label{eq:someeq2}
\end{align}
\end{proof}

If in the present setting $\be\vert_{\dOm}=\al\vert_{\dOm}$, then we can assume $\be=\al$ on $ \Om$ and the factors $\left|\frac{\be}{\al}\right|_\infty$ appearing in Proposition~\ref{prop:mixedterm} and Proposition~\ref{prop:varboundary} equal 1.\\

By specifying functions $\varphi$ and $\rho$ with the properties requested in Propositions~\ref{prop:mixedterm} and~\ref{prop:varboundary} we may now obtain an explicit upper bound for $C_\mu$. As in~\cite{brw} both functions will be chosen of the form $\psi \circ \rho_{\dOm}$ for some appropriate function $\psi$.\\
We use for $k,\gamma\in\mathbb{R}$ the function
\begin{equation}\label{eq:hfctn}
h:[0,\infty)\to\mathbb{R}, \ h(t):=\begin{cases}
	\cos(\sqrt{k}t)-\frac{\gamma}{\sqrt{k}}\sin(\sqrt{k}t) & k\ge0, \\
	\cosh(\sqrt{-k}t)-\frac{\gamma}{\sqrt{-k}}\sinh(\sqrt{-k}t) & k<0.
\end{cases}
\end{equation}

\begin{lemma}\label{lem:rhol2}
Let $k_1,k_2\in\mathbb{R}$ such that $\mathrm{Ric}\ge k_1 (d-1),\  \mathrm{sect} \le k_2$ and $\gamma_1,\gamma_2\in\mathbb{R}$ such that $\gamma_1 id \le \mathrm{I\!I} \le \gamma_2 id$ and let $h_i, i=1,2$ be as defined above in~\eqref{eq:hfctn} with $k=k_i$ and $\gamma=\gamma_i$. For $t_0\in (0,h_2^{-1}(0))$ we construct a function  $\varphi\in C^1(\Om)$ (depending on $t_0$) such that $\frac{\partial \varphi}{\partial N}\vert_{\dOm}=-1$ and $\nabla \varphi$ is Lipschitz continuous on $\Om$ and 
\small \begin{align*}
|\nabla \varphi|_{2,\be d\la}^2&=  \int_{\{\rho_{\dOm}\le t_0\}} \lp 1-\frac{\rho_{\dOm}}{t_0}\rp^2 \be d\la, \\
\left|\Delta \varphi + \frac{\nabla \be}{\be}\cdot\nabla\varphi \right|_{2,\be d\la}^2 &\le \int_{\{\rho_{\dOm} \le t_0\}} \biggl[ \lp \lp (d-1)\frac{h_2'}{h_2}(\rho_{\dOm}) -\frac{1}{t_0-\rho_{\dOm}}\rp^-\rp^2 
		+ \lp \lp (d-1)\frac{h_1'}{h_1}(\rho_{\dOm})  -\frac{1}{t_0-\rho_{\dOm}}\rp^+\rp^2 \\
		&+ 2\lp  (d-1)\frac{h_2'}{h_2}(\rho_{\dOm}) -\frac{1}{t_0-\rho_{\dOm}}\rp^- \frac{|\nabla \be|}{\be} +2\lp (d-1)\frac{h_1'}{h_1}(\rho_{\dOm})  -\frac{1}{t_0-\rho_{\dOm}}\rp^+ \frac{|\nabla \be|}{\be} \\
		&+ \frac{|\nabla \be|^2}{\be^2} \biggr] \lp 1-\frac{\rho_{\dOm}}{t_0}\rp^2 \be d\la.
\end{align*}
\end{lemma}

\begin{proof} It is easy to see that $h_2^{-1}(0)\le h_1^{-1}(0)$. Let $\rho_{\dOm}$ be the distance function to the boundary. By the Laplacian comparison theorem
\begin{align}
\Delta \rho_{\dOm} \le \frac{(d-1)h_1'}{h_1}(\rho_{\dOm})\text { on }  \{\rho_{\dOm} < h_1^{-1}(0)\}, \label{eq:laplacecomp1}\\
\Delta \rho_{\dOm} \ge \frac{(d-1)h_2'}{h_2}(\rho_{\dOm})\text{ on } \{\rho_{\dOm} < h_2^{-1}(0)\} \label{eq:laplacecomp2},
\end{align}
see~\cite{brw}.\\
Now for $t_0\in (0,h_2^{-1}(0))$ define
\begin{equation*}
\varphi:= \int_0^{\rho_{\dOm}} \lp 1-\frac{s}{t_0}\rp^+ ds.
\end{equation*}
It holds
\begin{equation*}
\nabla \varphi (x)= \begin{cases} \nabla \rho_{\dOm}(x)\cdot \lp 1-\frac{\rho_{\dOm}(x)}{t_0} \rp  &\rho_{\dOm}(x) \le t_0, \\
	0  &\text{ else},
	\end{cases}
\end{equation*}
and thus $\frac{\partial \varphi}{\partial N}\vert_{\dOm}=-1$ and $\nabla \varphi$ is Lipschitz continuous. Furthermore,
\begin{equation*}
\Delta \varphi (x)= \begin{cases}\Delta \rho_{\dOm}(x) \lp 1- \frac{\rho_{\dOm}(x)}{t_0}\rp -\frac{1}{t_0} & \rho_{\dOm}(x)\le t_0, \\
	0 &\text{ else}.
	\end{cases}
\end{equation*}
We then get
\begin{equation*}
\int_\Om |\nabla \varphi|^2 \be d\la = \int_{\{\rho_{\dOm}\le t_0\}} \lp 1-\frac{\rho_{\dOm}}{t_0}\rp^2 |\nabla \rho_{\dOm}|^2 \be d\la =  \int_{\{\rho_{\dOm}\le t_0\}} \lp 1-\frac{\rho_{\dOm}}{t_0}\rp^2 \be d\la.
\end{equation*}
Moreover,
\begin{equation*}
\int_\Om \lp \Delta \varphi + \frac{\nabla \be}{\be}\cdot \nabla\varphi\rp^2 \be d \la = \int_\Om \lp (\Delta \varphi)^2 + 2\Delta \varphi \frac{\nabla \be}{\be} \cdot \nabla \varphi + \frac{(\nabla \be\cdot\nabla\varphi)^2}{\be^2}\rp \be d\la.
\end{equation*}
By inequalities~\eqref{eq:laplacecomp1} and~\eqref{eq:laplacecomp2} on $\{\rho_{\dOm}\le t_0\}$
\begin{align*}
\Delta\varphi \ge  (d-1)\frac{h_2'}{h_2}(\rho_{\dOm}) \lp 1-\frac{\rho_{\dOm}}{t_0} \rp -\frac{1}{t_0}\ &\Rightarrow\ (\Delta\varphi)^- \le  \lp  (d-1)\frac{h_2'}{h_2}(\rho_{\dOm}) \lp 1-\frac{\rho_{\dOm}}{t_0} \rp -\frac{1}{t_0}\rp^-,\\
\Delta\varphi \le  (d-1)\frac{h_1'}{h_1}(\rho_{\dOm}) \lp 1-\frac{\rho_{\dOm}}{t_0} \rp -\frac{1}{t_0}\ &\Rightarrow\ (\Delta\varphi)^+ \le \lp (d-1)\frac{h_1'}{h_1}(\rho_{\dOm})   \lp 1-\frac{\rho_{\dOm}}{t_0} \rp -\frac{1}{t_0}\rp^+.
\end{align*}
Thus
\begin{align*}
&\int_\Om  \lp (\Delta \varphi)^2 + 2\Delta \varphi \frac{\nabla \be}{\be} \cdot \nabla \varphi + \frac{(\nabla \be\cdot\nabla\varphi)^2}{\be^2}\rp \be d\la \\
	&\le \int_\Om \lp \lp(\Delta \varphi)^-\rp^2 + \lp(\Delta \varphi)^+\rp^2 + 2(\Delta \varphi)^- \frac{|\nabla \be|}{\be} |\nabla \varphi| +2(\Delta \varphi)^+ \frac{|\nabla \be|}{\be} |\nabla \varphi| + \frac{(\nabla \be\cdot\nabla\varphi)^2}{\be^2}\rp \be d\la \\
	&\le \int_{\{\rho_{\dOm}\le t_0\}} \biggl( \lp  \lp (d-1)\frac{h_2'}{h_2}(\rho_{\dOm}) \lp 1-\frac{\rho_{\dOm}}{t_0} \rp -\frac{1}{t_0}\rp^-\rp^2 
		+ \lp \lp (d-1)\frac{h_1'}{h_1}(\rho_{\dOm})   \lp 1-\frac{\rho_{\dOm}}{t_0} \rp -\frac{1}{t_0}\rp^+\rp^2 \\
		&+ 2\lp  (d-1)\frac{h_2'}{h_2}(\rho_{\dOm}) \lp 1-\frac{\rho_{\dOm}}{t_0} \rp -\frac{1}{t_0}\rp^- \frac{|\nabla \be|}{\be} \lp 1- \frac{\rho_{\dOm}}{t_0}\rp  \\
		&+2\lp (d-1)\frac{h_1'}{h_1}(\rho_{\dOm})   \lp 1-\frac{\rho_{\dOm}}{t_0} \rp -\frac{1}{t_0}\rp^+ \frac{|\nabla \be|}{\be} \lp 1- \frac{\rho_{\dOm}}{t_0}\rp 
		+ \frac{|\nabla \be|^2}{\be^2} \lp 1-\frac{\rho_{\dOm}}{t_0}\rp^2 \biggr) \be d\la \\
	&= \int_{\{\rho_{\dOm} \le t_0\}} \biggl[ \lp \lp (d-1)\frac{h_2'}{h_2}(\rho_{\dOm}) -\frac{1}{t_0-\rho_{\dOm}}\rp^-\rp^2 
		+ \lp \lp (d-1)\frac{h_1'}{h_1}(\rho_{\dOm})  -\frac{1}{t_0-\rho_{\dOm}}\rp^+\rp^2 \\
		&+ 2\lp  (d-1)\frac{h_2'}{h_2}(\rho_{\dOm}) -\frac{1}{t_0-\rho_{\dOm}}\rp^- \frac{|\nabla \be|}{\be} +2\lp (d-1)\frac{h_1'}{h_1}(\rho_{\dOm})  -\frac{1}{t_0-\rho_{\dOm}}\rp^+ \frac{|\nabla \be|}{\be} \\
		&+ \frac{|\nabla \be|^2}{\be^2} \biggr] \lp 1-\frac{\rho_{\dOm}}{t_0}\rp^2 \be d\la.
\end{align*}
\end{proof}

In the previous Lemma $t_0\in (0,h_2^{-1}(0))$ may be chosen to either optimise $\left|\Delta \varphi + \frac{\nabla \be}{\be}\cdot\nabla\varphi \right|_{2,\be\la}^2$ or $|\nabla \varphi|_{2,\be\la}^2$.

\begin{lemma}\label{lem:rho}
Let $k_2\in\mathbb{R}$ such that $\mathrm{sect} \le k_2$ and $\gamma_2\in\mathbb{R}$ such that $\mathrm{I\!I}\le \gamma_2 id$. Then $k_2>-\gamma_2^2$, and for all $\eps>0$ there exists a function $\rho\in C^1(\Om)$ such that $\frac{\partial \rho}{\partial N}\vert_{\dOm}=-1$ and $\nabla \rho$ is Lipschitz continuous on $\Om$ and
\begin{align*}
|\nabla \rho|_\infty &\le 1, \\
\left| \lp \Delta \rho + \frac{\nabla \be}{\be}\cdot \nabla \rho \rp^- \right|_\infty &\le \inf_{t_1\in(0,h_2^{-1}(0))} \sup_{\{\rho_{\dOm}\le t_1\}} \lp 1- \frac{\rho_{\dOm}}{t_1} \rp \lp (d-1)\frac{h_2'}{h_2}(\rho_{\dOm}) - \frac{1}{t_1-\rho_{\dOm}} - \frac{|\nabla \be|}{\be} \rp^- +\eps.
\end{align*}
\end{lemma}

\begin{proof} 
Let  $h_2$ be the function defined in~\eqref{eq:hfctn} with $k=k_2$ and $\gamma=\gamma_2$. It holds $k_2>-\gamma_2^2$, see~\cite{brw}.
Let $t_1\in(0,h_2^{-1}(0))$ to be chosen later. 
By~\eqref{eq:laplacecomp2}
\begin{equation}\label{eq:laplacea}
\Delta \rho_{\dOm} \ge (d-1)\frac{h_2'}{h_2}(\rho_{\dOm}) 
 \text{ on } \{\rho_{\dOm} \le t_1\}.
\end{equation}
Now define
\begin{equation*}
\rho:= \int_0^{\rho_{\dOm}} \lp 1-\frac{s}{t_1}\rp^+ ds.
\end{equation*}
It holds 
\begin{equation*}
\nabla \rho (x)= \begin{cases} \nabla \rho_{\dOm}(x)\cdot \lp 1-\frac{\rho_{\dOm}(x)}{t_1} \rp  &\rho_{\dOm}(x) \le t_1, \\
	0  &\text{ else}, 
	\end{cases}
\end{equation*}
and thus $\frac{\partial \rho}{\partial N}\vert_{\dOm}=-1$, $|\nabla \rho|_\infty\le 1$ and $\nabla \rho$ is Lipschitz continuous. Furthermore,
\begin{equation*}
\Delta \rho (x)= \begin{cases}\Delta \rho_{\dOm}(x) \lp 1- \frac{\rho_{\dOm}(x)}{t_1}\rp -\frac{1}{t_1} & \rho_{\dOm}(x)\le t_1,\\
	0 & \text{ else},
	\end{cases}
\end{equation*}
and thus by inequality~\eqref{eq:laplacea} on $\{\rho_{\dOm}\le t_1\}$
\begin{align*}
\Delta \rho + \frac{\nabla \be}{\be}\cdot \nabla \rho &\ge (d-1)\frac{h_2'}{h_2}(\rho_{\dOm})\lp 1-\frac{\rho_{\dOm}}{t_1} \rp - \frac{1}{t_1} - \frac{|\nabla \be|}{\be} |\nabla \rho|\\	
	&\ge  (d-1)\frac{h_2'}{h_2}(\rho_{\dOm})\lp 1-\frac{\rho_{\dOm}}{t_1} \rp - \frac{1}{t_1} - \frac{|\nabla \be|}{\be} \lp1-\frac{\rho_{\dOm}}{t_1}\rp \\	
	&= \lp 1- \frac{\rho_{\dOm}}{t_1} \rp \lp (d-1)\frac{h_2'}{h_2}(\rho_{\dOm}) - \frac{1}{t_1-\rho_{\dOm}} - \frac{|\nabla \be|}{\be} \rp.
\end{align*}
This implies on $\{\rho_{\dOm}\le t_1\}$
\begin{equation*}
\lp \Delta \rho + \frac{\nabla \be}{\be}\cdot \nabla \rho \rp^- \le \lp 1- \frac{\rho_{\dOm}}{t_1} \rp \lp (d-1)\frac{h_2'}{h_2}(\rho_{\dOm}) - \frac{1}{t_1-\rho_{\dOm}} - \frac{|\nabla \be|}{\be} \rp^-.
\end{equation*}
We can still choose $t_1\in(0,h_2^{-1}(0))$ to obtain for arbitrary $\eps>0$ \begin{small}
\begin{equation*}
\left| \lp \Delta \rho + \frac{\nabla \be}{\be}\cdot \nabla \rho \rp^- \right|_\infty \le \inf_{t_1\in(0,h_2^{-1}(0))} \sup_{\{x|\rho_{\dOm}(x)\le t_1\}} \lp 1- \frac{\rho_{\dOm}(x)}{t_1} \rp \lp (d-1)\frac{h_2'}{h_2}(\rho_{\dOm}) - \frac{1}{t_1-\rho_{\dOm}} - \frac{|\nabla \be|}{\be} \rp^-(x) +\eps.
\end{equation*}\end{small}
\end{proof}

In the previous two Lemmata we bound $\lp \Delta \rho + \frac{\nabla \be}{\be}\cdot \nabla \rho \rp$ using the Laplacian comparison theorem. It would seem convenient to use a weighted Laplacian comparison theorem on manifolds with boundary instead of the non-weighted version used above. One such result giving an upper bound on the weighted Laplacian of the distance to the boundary function has been shown in~\cite{sakurei} under the assumption of a lower bound on the $n$-weighted Ricci curvature. As we are not aware of an analogous result for the other direction we refrain from using the result in~\cite{sakurei} so as to keep the setting coherent and not mix different types of curvature assumptions.\\

Inserting $\varphi$ and $\rho$ as defined in Lemma~\ref{lem:rhol2} and Lemma~\ref{lem:rho} in Proposition~\ref{prop:mixedterm} and Proposition~\ref{prop:varboundary} we now get explicit constants $K_1, K_2$ and $K_{\dOm,\Om}$ in terms of bounds on sectional and Ricci curvature and second fundamental form on the boundary. We state these in the following Theorem. Via Proposition~\ref{prop:varinterpol} we thus obtain an upper bound on $C_\mu$.

\begin{theorem}\label{thm:explconst}
Let $k_1,k_2\in\mathbb{R}$ such that $\mathrm{Ric} \ge (d-1)k_1,  \mathrm{sect} \le k_2$ and $\gamma_1,\gamma_2\in\mathbb{R}$ such that $\gamma_1 id \le \mathrm{I\!I} \le \gamma_2 id$. Then the assumptions in Proposition~\ref{prop:varinterpol} are fulfilled with \small
\begin{align*}
K_1 &= \frac{A}{B^2}\left|\frac{\be}{\al}\right|_\infty \inf_{t_0\in(0,h_2^{-1}(0))} \inf_{\eps\in(0,\infty)} \Biggl((1+\eps)  C_{\laa}\int_{\{\rho_{\dOm}\le t_0\}} \biggl[ \lp \lp (d-1)\frac{h_2'}{h_2}(\rho_{\dOm}) -\frac{1}{t_0-\rho_{\dOm}}\rp^-\rp^2 \\
		&+ \lp \lp (d-1)\frac{h_1'}{h_1}(\rho_{\dOm})  -\frac{1}{t_0-\rho_{\dOm}}\rp^+\rp^2 + 2\lp  (d-1)\frac{h_2'}{h_2}(\rho_{\dOm}) -\frac{1}{t_0-\rho_{\dOm}}\rp^- \frac{|\nabla \be|}{\be}\\
		& +2\lp (d-1)\frac{h_1'}{h_1}(\rho_{\dOm})  -\frac{1}{t_0-\rho_{\dOm}}\rp^+ \frac{|\nabla \be|}{\be} + \frac{|\nabla \be|^2}{\be^2} \biggr] \lp 1-\frac{\rho_{\dOm}}{t_0}\rp^2 \be d\la 
 + (1+\eps^{-1}) \int_{\{\rho_{\dOm}\le t_0\}} \lp 1-\frac{\rho_{\dOm}}{t_0}\rp^2 \be d\la \Biggr) ,\\
K_2 &= 0,\\
K_{\dOm,\Om} &=  \frac{A}{B}\left|\frac{\be}{\al}\right|_\infty \lp C_{\laa}\inf_{t_1\in(0,h_2^{-1}(0))} \sup_{\{\rho_{\dOm}\le t_1\}} \lp 1- \frac{\rho_{\dOm}}{t_1} \rp \lp (d-1)\frac{h_2'}{h_2}(\rho_{\dOm}) - \frac{1}{t_1-\rho_{\dOm}} - \frac{|\nabla \be|}{\be} \rp^- + 2C_{\laa}^{(1/2)}\rp.
\end{align*}
\end{theorem}

As explained in Remark~\ref{rem:steklov} bounding the optimal $K_{\dOm,\Om}$ fulfilling inequality~\eqref{eq:varboundary} from above corresponds to bounding the first non-trivial doubly weighted Steklov eigenvalue $\sigma_{\mu,1}$ from below, thus we may directly note the following corollary:

\begin{corollary}\label{cor:steklovbound}
Let $k_2\in\mathbb{R}$ such that $\mathrm{sect}\le k_2$ and $\gamma_2\in\mathbb{R}$ such that $\mathrm{I\!I} \le \gamma_2 id$. Then for the first non-trivial doubly weighted Steklov eigenvalue $\sigma_{\mu,1}$ of $\Om$ it holds that
\begin{equation*}
	\sigma_{\mu,1} \ge \lp \left|\frac{\be}{\al}\right|_\infty \lp C_{\laa}\inf_{t_1\in(0,h_2^{-1}(0))} \sup_{\{\rho_{\dOm}\le t_1\}} \lp 1- \frac{\rho_{\dOm}}{t_1} \rp \lp (d-1)\frac{h_2'}{h_2}(\rho_{\dOm}) - \frac{1}{t_1-\rho_{\dOm}} - \frac{|\nabla \be|}{\be} \rp^- + 2C_{\laa}^{(1/2)}\rp \rp^{-1}.
\end{equation*}
\end{corollary}

\smallskip

\subsection{Sticky reflection without boundary diffusion}\label{ssec:pcwobd}

In this subsection we consider the case $\delta=0$ without weighted boundary diffusion. Note that the spectrum (but not spectral gap) of $\hat{L}$ has previously been studied in the case of constant weights in~\cite{belowgilles}.\\
Here we will show that a Poincar\'{e} inequality holds for Brownian motion with sticky reflection doubly weighted by $\mu$ on $\Om$ and give an upper bound on the Poincar\'{e} constant $\hat{C}_\mu$. The bound will depend on the Poincar\'{e} constant $C_{\laa}$.
The main result providing an upper bound on $C_\mu$ in this setting will be via the combination of Proposition~\ref{prop:pcsrbm} and Theorem~\ref{thm:explconst}.

\begin{proposition}\label{prop:expcwobd}
There is a constant $\hat{C}_\mu\in\mathbb{R}$ such that for all $f\in C^1(\Om)$
\begin{equation*}
\mathrm{Var}_{\mu}(f) \le \hat{C}_\mu \hat{\eE}_\mu(f).
\end{equation*}
\end{proposition}
\begin{proof}
Assume the claim is false, i.e.\ there is a sequence $(f_k)_{k\in\mathbb{N}}\subset C^1(\Om)$ such that 
\begin{equation}\label{eq:PIexistencewsr}
\mathrm{Var}_{\mu}(f_k)>k \hat{\eE}_\mu(f_k).
\end{equation} Assume without loss of generality that $\int_\Om f_k d\mu =0$ and $\mathrm{Var}_{\mu}(f_k)=1$ for all $k\in\mathbb{N}$. Then inequality~\eqref{eq:PIexistencewsr} implies that
\begin{equation*}
\hat{\eE}_\mu(f_k)<\frac{1}{k}\ \Leftrightarrow\ \int_\Om |\nabla f_k|^2 \al d\la <\frac{1}{k}.
\end{equation*}
Furthermore, $\mathrm{Var}_{\mu}(f_k)=1$ together with $\int_\Om f_k d\mu =0$ implies
\begin{equation*}
|f_k|^2_{2,\laa} \le \frac{1}{A} + \frac{B}{A}.
\end{equation*}
Thus, $(f_k)_k$ is a bounded sequence in $W^{1,2}(\Om,\al d\la)$. The following is up to taking subsequences: As $W^{1,2}(\Om,\al d\la)$ is a Hilbert space $(f_k)_k$ converges weakly to an element $f$ in $W^{1,2}(\Om,\al d\la)$. As $W^{1,2}(\Om,\al d\la)$ is compactly embedded in $L^2(\Om,\al d\la)$ and $L^{2}(\dOm,\be d\si)$, $(f_k)_k$ converges to $f$ in $L^2(\Om,\al d\la)$ and (by restriction to the boundary) in $L^2(\dOm,\be d\si)$. Thus, 
\begin{equation*}
1= \lim_{k\to\infty} \mathrm{Var}_{\mu}(f_k) = \lim_{k\to\infty} |f_k|^2_{2,\mu} = |f|^2_{2,\mu}.
\end{equation*} 
Furthermore, $\int_\Om f d\mu=0$. Since $f\in W^{1,2}(\Om,\al d\la)$ with $\nabla f =0$, $f$ has to be constant $\la$ almost surely. If $f\equiv c$ for $c\neq0$, then also $f\vert_{\dOm}\equiv c$ and $\int_\Om f d\mu=c\neq 0$. Thus, $f\equiv 0$ which is a contradiction to $|f|_{2,\mu}=1$.
\end{proof}

\begin{proposition}\label{prop:pcsrbm}
Assume there are constants $K_1,K_{\dOm,\Om}$ such that for all $f\in C^1(\Om)$
\begin{align*}
\var_{\sib}(f) &\le K_{\dOm,\Om} \int_\Om |\nabla f|^2 d\laa, \\
\lp \int_\Om f d\laa - \int_{\dOm} f d\sib \rp^2 &\le K_1 \int_\Om |\nabla f|^2d\laa, 
\end{align*}
then
\begin{equation*}
\hat{C}_\mu \le \lp C_{\laa} + \frac{BK_{\dOm,\Om}}{A} + BK_1 \rp.
\end{equation*}
\end{proposition}
\begin{proof}
We decompose the variance with respect to $\mu$ as in the proof of Proposition~\ref{prop:varinterpol}
\begin{align*}
\var_\mu(f) &=  A\var_{\laa}(f) + B\var_{\sib}(f) + AB\lp \int_\Om f d\laa - \int_{\dOm} f d\sib \rp^2 \\
	&\le \lp C_{\laa} + \frac{BK_{\dOm,\Om}}{A} + BK_1\rp \int_\Om |\nabla f|^2 \al d\la.
\end{align*}
Thus the statement of the Proposition follows.
\end{proof}

We can use $K_1$ and $K_{\dOm,\Om}$ from Theorem~\ref{thm:explconst} above because we managed to prove the statements there with $K_2=0$ and thus also obtain an upper bound on $\hat{C}_\mu$.

\begin{remark}\label{rem:comp}
We compare the results obtained above with the results in~\cite{mvr} and~\cite{brw}: In~\cite{mvr,brw} a special case of the present setting is considered. It corresponds to 
\begin{equation*}
	\al= \frac{\overline{\al}}{|\Om|},\ \be=\frac{1-\overline{\al}}{|\dOm|},
\end{equation*}
where $\overline{\al}\in(0,1)$ is constant. In this special case inequalities~\eqref{eq:varboundary} and~\eqref{eq:mixedterm} simplify to the conditions used in~\cite{mvr,brw} and the explicit values of $K_1,K_2,K_{\dOm,\Om}$ computed in Propositions~\ref{prop:mixedterm} and \ref{prop:varboundary} and finally Theorem~\ref{thm:explconst} equal the corresponding values in~\cite{brw}. Thus the overall upper bounds on $C_\mu$ and $\hat{C}_\mu$ obtained here equal the corresponding upper bounds obtained in~\cite{brw}.
\end{remark}

\section{Logarithmic Sobolev Inequality}\label{sec:LSI}
\subsection{Sticky-reflecting boundary diffusion}\label{ssec:lsiwbd}
Again we first consider the case $\delta=1$ including weighted boundary diffusion and discuss the case $\delta=0$ without boundary diffusion in subsection~\ref{ssec:lsiwobd}. 
We will give upper bounds on the logarithmic Sobolev constant $L_\mu$ for Brownian motion with sticky-reflecting boundary diffusion doubly weighted by $\mu$ on $\Om$. The bound will depend on the Poincar\'{e} and logarithmic Sobolev constants $C_{\laa}, L_{\laa}, C_{\sib},L_{\sib}$.
The main result providing an upper bound on $L_\mu$ in this setting will be obtained via combination of Proposition~\ref{prop:logsobinterpol}, Theorem~\ref{thm:explconst} and Theorem~\ref{thm:wbdrlogsob}.\\
We first prove an analogue of Proposition~\ref{prop:varinterpol} which is also an adaptation of a result from~\cite{brw} to the weighted setting. Note that assumptions \eqref{eq:cond1} and \eqref{eq:cond3} in the next Proposition are exactly the same as in Proposition~\ref{prop:varinterpol} for the Poincar\'{e} inequality.

\begin{proposition}\label{prop:logsobinterpol}
Assume there are constants $K_{\dOm,\Om}, L_{\dOm,\Om},K_1, K_2$ such that for all $f\in C^1(\Om)$
\begin{align}
\var_{\sib}(f) &\le K_{\dOm,\Om} \int_\Om |\nabla f|^2 d\laa, \label{eq:cond1}\\
\ent_{\sib}(f^2) &\le L_{\dOm,\Om} \int_\Om |\nabla f|^2 d\laa, \label{eq:ent}\\
\lp \int_\Om f d\laa - \int_{\dOm} f d\sib \rp^2 &\le K_1 \int_\Om |\nabla f|^2 d\laa + K_2 \int_{\dOm} |\nabla^\tau f|^2 d\sib,\label{eq:cond3}
\end{align} 
then 
\begin{align}\label{eq:logsobupperbound}
L_\mu \le \inf_{s,t\in[0,1]} \max\Bigl(&L_{\laa} + s L_{\dOm,\Om} \frac{B}{A} + B\frac{\ln(B)-\ln(A)}{B-A} \lp C_{\laa} + t K_{\dOm,\Om} + K_1\rp, \notag \\
		&(1-s)L_{\sib} + A\frac{\ln(B)-\ln(A)}{B-A} \lp (1-t)C_{\sib} + K_2\rp\Bigr).
\end{align}
\end{proposition}

\begin{proof}
\allowdisplaybreaks
We may write $\mu=\al\la+\be\si = A\laa+B\sib$ as a mixture or more specifically convex combination of the two probability measures $\laa$ and $\sib$. Thus the decomposition of the entropy with respect to the mixture of two measures as well as an optimal logarithmic Sobolev inequality for Bernoulli measures as described in~\cite[section 4]{chafai} may be applied to $\mu$, which is the first step of the following computation
\begin{align*}
\ent_{\mu}(f^2) &\le A \ent_{\laa}(f^2) + B \ent_{\sib}(f^2)  \\
	&+AB\frac{\ln(B)-\ln(A)}{B-A}\lp \var_{\laa}(f)+\var_{\sib}(f) + \lp \int_\Om f d\laa - \int_{\dOm} f d\sib \rp^2 \rp \\
	&\le  \lp L_{\laa} + s L_{\dOm,\Om} \frac{B}{A} + AB\frac{\ln(B)-\ln(A)}{B-A} \lp \frac{C_{\laa}}{A} + t \frac{K_{\dOm,\Om}}{A} + \frac{K_1}{A}\rp \rp \int_\Om |\nabla f|\al 
		d\la\\	
	&+ \lp (1-s)L_{\sib} + AB\frac{\ln(B)-\ln(A)}{B-A} \lp (1-t)\frac{C_{\sib}}{B} + \frac{K_2}{B}\rp \rp \int_{\dOm} |\nabla^\tau f|^2 \be d\si\\
	&=  \lp L_{\laa} + s L_{\dOm,\Om} \frac{B}{A} + B\frac{\ln(B)-\ln(A)}{B-A} \lp C_{\laa} + t K_{\dOm,\Om} + K_1\rp \rp \int_\Om |\nabla f|\al 
		d\la\\	
	&+ \lp (1-s)L_{\sib} + A\frac{\ln(B)-\ln(A)}{B-A} \lp (1-t)C_{\sib} + K_2\rp \rp \int_{\dOm} |\nabla^\tau f|^2 \be d\si.
\end{align*}
Thus
\begin{align*}
L_\mu \le \inf_{s,t\in[0,1]} \max\Bigl(&L_{\laa} + s L_{\dOm,\Om} \frac{B}{A} + B\frac{\ln(B)-\ln(A)}{B-A} \lp C_{\laa} + t K_{\dOm,\Om} + K_1\rp, \\
		&(1-s)L_{\sib} + A\frac{\ln(B)-\ln(A)}{B-A} \lp (1-t)C_{\sib} + K_2\rp\Bigr).
\end{align*}
\end{proof}
Note that in analogy with \cite[Remark 4.1]{brw} the infimum in inequality~\eqref{eq:logsobupperbound} can be determined explicitly. 

\subsection{Sticky reflection without boundary diffusion}\label{ssec:lsiwobd}
In this subsection we consider the case $\delta=0$ without weighted boundary diffusion. We will give upper bounds on the logarithmic Sobolev constant $\hat{L}_\mu$ for Brownian motion with sticky-reflection doubly weighted by $\mu$ on $\Om$. The bound will depend on the Poincar\'{e} and logarithmic Sobolev constants $C_{\laa}$ and $L_{\laa}$. The main result providing an upper bound on $\hat{L}_\mu$ in this setting will be obtained via the combination of Proposition~\ref{prop:logsobinterpolwobd}, Theorem~\ref{thm:explconst} and Theorem~\ref{thm:wbdrlogsob}.

\begin{proposition}\label{prop:logsobinterpolwobd}
Assume there are constants $K_{\dOm,\Om}, L_{\dOm,\Om},K_1$ such that for all $f\in C^1(\Om)$
\begin{align*}
\var_{\sib}(f) &\le K_{\dOm,\Om} \int_\Om |\nabla f|^2 d\laa,\\
\ent_{\sib}(f^2) &\le L_{\dOm,\Om} \int_\Om |\nabla f|^2 d\laa,\\
\lp \int_\Om f d\laa - \int_{\dOm} f d\sib \rp^2 &\le K_1 \int_\Om |\nabla f|^2 d\laa,
\end{align*} 
then 
\begin{equation*}
\hat{L}_\mu \le  \lp L_{\laa} + \frac{B}{A}L_{\dOm,\Om} + B\frac{\ln(B)-\ln(A)}{B-A} \lp C_{\laa} + K_{\dOm,\Om} + K_1\rp \rp.
\end{equation*}
\end{proposition}

\begin{proof}
As in the previous proof we write $\mu=\al\la+\be\si = A\laa+B\sib$ as a convex combination of the two probability measures $\laa$ and $\sib$, decompose the entropy and use an optimal logarithmic Sobolev inequality for Bernoulli measures as described in~\cite[section 4]{chafai}:
\begin{align*}
\ent_\mu(f^2) &\le A \ent_{\laa}(f^2) + B \ent_{\sib}(f^2) \\
	&+ AB\frac{\ln(B)-\ln(A)}{B-A}\lp \var_{\laa}(f)+\var_{\sib}(f) + \lp \int_\Om f d\laa - \int_{\dOm} f d\sib\rp^2 \rp \\
	&\le \lp L_{\laa} + \frac{B}{A}L_{\dOm,\Om} + AB\frac{\ln(B)-\ln(A)}{B-A} \lp \frac{C_{\laa}}{A} + \frac{K_{\dOm,\Om}}{A} + \frac{K_1}{A}\rp \rp \int_\Om |\nabla f|^2 \al d\la \\
	&= \lp L_{\laa} + \frac{B}{A}L_{\dOm,\Om} + B\frac{\ln(B)-\ln(A)}{B-A} \lp C_{\laa} + K_{\dOm,\Om} + K_1\rp \rp \int_\Om |\nabla f|^2 \al d\la.
\end{align*}
\end{proof}

Analogously to Remark~\ref{rem:comp} in the special case of constant weights Propositions~\ref{prop:logsobinterpol} and \ref{prop:logsobinterpolwobd} coincide with the known results from~\cite{brw}.\\
The main focus of the next section will be to find a constant $L_{\dOm,\Om}$ fulfilling condition~\eqref{eq:ent}. Note that it is not trivial that condition~\eqref{eq:ent} holds at all. Via Propositions~\ref{prop:logsobinterpol} and~\ref{prop:logsobinterpolwobd} as well as Theorem~\ref{thm:explconst} we will then obtain upper bounds on $L_\mu$ and $\hat{L}_\mu$.

\section{Weighted Boundary-Interior Inequalities}\label{sec:bii}

In subsection~\ref{ssec:sobtraceineq} we collect some further boundary-interior inequalities that are of independent interest. These will eventually be used in subsection~\ref{ssec:boundarytracelogsob} in order to bound the logarithmic Sobolev constant. 

\subsection{Weighted Sobolev-Poincar\'{e} trace inequalities} \label{ssec:sobtraceineq}

We first recall the following statement that was already obtained in the proof of Proposition~\ref{prop:varboundary} (see inequalities~\ref{eq:someeq1} and~\ref{eq:someeq2}).

\begin{proposition}\label{anotherprop}
Let $\rho\in C^1(\Om)$ such that $\frac{\partial \rho}{\partial N}\vert_{\dOm}=- 1$ and $\nabla \rho$ is Lipschitz continuous on $\Om$. Then it holds \small
\begin{align*}
	\int_{\dOm} f^2 d \sib &\le \frac{A}{B}\left|\lp \Delta \rho +\frac{\nabla \be}{\be} \cdot \nabla \rho \rp^-\right|_\infty \left|\frac{\be}{\al}\right|_\infty \int_\Om f^2 d \laa + \frac{2A}{B}|\nabla \rho|_\infty \left|\frac{\be}{\al}\right|_\infty \lp \int_\Om f^2 d\laa \int_\Om |\nabla f|^2 d\laa \rp^{1/2}\\
	&\le \frac{A}{B}\left|\frac{\be}{\al}\right|_\infty\lp C_{\laa}\left|\lp \Delta \rho +\frac{\nabla \be}{\be} \cdot \nabla \rho \rp^-\right|_\infty   + 2C_{\laa}^{(1/2)}|\nabla \rho|_\infty \rp \int_\Om |\nabla f|^2 d\laa.
\end{align*}
for all $f\in C^1(\Om)$ with $\int_\Om f d\laa =0$.
\end{proposition}

Note that the assumption $\int_\Om f d\laa =0$ was only used for the second inequality.\\
From this the next corollary follows immediately and gives an alternative upper bound for $K_1$ other than the one already obtained in Proposition~\ref{prop:mixedterm}.

\begin{corollary}
For $\rho\in C^1(\Om)$ such that $\frac{\partial \rho}{\partial N}\vert_{\dOm}=- 1$ and $\nabla \rho$ is Lipschitz continuous on $\Om$ inequality \eqref{eq:mixedterm} holds with $K_2=0$ and
\begin{equation*}
	K_1 = \frac{A}{B}\left|\frac{\be}{\al}\right|_\infty\lp C_{\laa}\left|\lp \Delta \rho +\frac{\nabla \be}{\be} \cdot \nabla \rho \rp^-\right|_\infty   + 2C_{\laa}^{(1/2)}|\nabla \rho|_\infty \rp.
\end{equation*}
\end{corollary}

Additionally the following remark deduces a bound on the doubly weighted Sobolev trace operator from Proposition~\ref{anotherprop}:

\begin{remark}\label{rem:Sobtracebound}
For any $f\in C^1(\Om)$\begin{small}
\begin{align*}
	\int_{\dOm} f^2 \be d\si &\le \left|\lp \Delta \rho +\frac{\nabla \be}{\be} \cdot \nabla \rho \rp^-\right|_\infty \left|\frac{\be}{\al}\right|_\infty \int_\Om f^2 \al d \la + 2|\nabla \rho|_\infty \left|\frac{\be}{\al}\right|_\infty \lp \int_\Om f^2 \al d\la \int_\Om |\nabla f|^2 \al d\la \rp^{1/2}\\
	&\le \left|\frac{\be}{\al}\right|_\infty \lp \left|\lp \Delta \rho +\frac{\nabla \be}{\be} \cdot \nabla \rho \rp^-\right|_\infty  + |\nabla \rho|_\infty \rp \int_\Om f^2 \al d\la + \lp |\nabla \rho|_\infty \left|\frac{\be}{\al}\right|_\infty \rp \int_\Om |\nabla f|^2 \al d\la.
\end{align*}
\end{small} Thus for $K_3:=\left|\frac{\be}{\al}\right|_\infty \lp \left|\lp \Delta \rho +\frac{\nabla \be}{\be} \cdot \nabla \rho \rp^-\right|_\infty  + |\nabla \rho|_\infty \rp$
\begin{equation*}
	|f|_{2,\be d\si} \le \sqrt{K_3} |f|_{W^{1,2}(\Om,\al d\la)}.
\end{equation*}
\end{remark}

$W^{1,2}(\Om, \al d\la)$ is the completion of smooth functions whose derivatives up to degree 1 are in $L^2(\Om, \al d\la)$ and thus the inequality also holds for all functions in $W^{1,2}(\Om, \al d\la)$. Therefore by using an explicit value for $K_3$ as obtained from Lemma~\ref{lem:rho} we may give an explicit upper bound for the norm of the trace operator $\vert_{\dOm} :W^{1,2}(\Om, \al d\la) \to L^2(\dOm, \be d\si)$:

\begin{proposition}\label{prop:sobtracebound}
Let $k_2\in\mathbb{R}$ such that $\mathrm{sect} \le k_2$ and $\gamma_2\in\mathbb{R}$ such that $\mathrm{I\!I}\le \gamma_2 id$. Then the norm of the Sobolev trace operator $\vert_{\dOm} :W^{1,2}(\Om, \al d\la) \to L^2(\dOm, \be d\si)$ is bounded from above by
\begin{equation*}
	\lp \left|\frac{\be}{\al}\right|_\infty \lp \inf_{t_1\in(0,h_2^{-1}(0))} \sup_{\{\rho_{\dOm}\le t_1\}} \lp 1- \frac{\rho_{\dOm}}{t_1} \rp \lp (d-1)\frac{h_2'}{h_2}(\rho_{\dOm}) - \frac{1}{t_1-\rho_{\dOm}} - \frac{|\nabla \be|}{\be} \rp^-  + 1 \rp\rp^{1/2}.
\end{equation*}	
\end{proposition}

\bigskip
In the following proposition we show a boundary-interior version of the doubly weighted Sobolev-Poincar\'{e} inequality:

\begin{proposition}\label{prop:boundaryinterior2}
Let $(\Om,g)$ be a smooth, compact Riemannian manifold of dimension $d\ge 3$, with a connected boundary.
For any $\rho\in C^1(\Om)$ such that $\frac{\partial \rho}{\partial N}\vert_{\dOm}=-1$ and $\nabla \rho$ is Lipschitz continuous on $\Om$ it holds for all $f\in C^1(\Om)$ with $\int_{\Om} f d\laa =0$ and $p\in \left[2,\frac{2d-2}{d-2}\right]$\small
\begin{equation*}
\lp \int_{\dOm} |f|^p d\sib \rp^{2/p} \le \lp \lp \lp  p |\nabla \rho|_\infty \rp^{2} \left|\frac{\be}{\al}\right|_\infty \frac{A}{B}\lp C_{2(p-1),2}^{\be,\al}\rp^{2(p-1)}\rp^{1/p} + \left|\lp \Delta \rho+\nabla \rho \cdot \frac{\nabla \be}{\be} \rp^-\right|_\infty ^{2/p} \lp C_{p,2}^{\be,\al}\rp^2 \rp \int_\Om |\nabla f|^2 d\laa.
\end{equation*}
\end{proposition}

\begin{proof}
\allowdisplaybreaks
We may calculate as in the previous proofs to obtain
\begin{align*}
 \Bigl(& \int_{\dOm} |f|^p d\sib \Bigr)^{2/p} = \lp -\int_{\dOm} |f|^p \frac{\partial \rho}{\partial N} \frac{\be}{B} d\si \rp^{2/p} \\
	&= \lp  - \int_\Om \nabla \rho \cdot \frac{\nabla(|f|^p \be)}{B}d\la -\int_\Om |f|^p \frac{\be}{B}\Delta \rho d\la\rp^{2/p} \\
	&= \lp  - p\int_\Om |f|^{p-1}\nabla \rho \cdot \nabla f \frac{\be}{B}d\la -\int_\Om |f|^p \frac{\be}{B}\lp \Delta \rho+\nabla \rho\cdot\frac{\nabla \be}{\be} \rp d\la\rp^{2/p} \\
	&\le \lp |\nabla \rho|_\infty p \int_\Om |f|^{p-1} |\nabla f| \frac{\be}{B} d\la + \left|\lp \Delta \rho+\nabla \rho \cdot \frac{\nabla \be}{\be} \rp^-\right|_\infty \int_\Om |f|^p \frac{\be}{B} d\la  \rp^{2/p} \\
	&\le \lp  |\nabla \rho|_\infty p \lp \int_\Om |f|^{2(p-1)} \frac{\be}{B}d\la \rp ^{1/2} \lp \int_\Om |\nabla f|^2 \frac{\be}{B}d\la \rp ^{1/2} + \left|\lp \Delta \rho+\nabla \rho \cdot \frac{\nabla \be}{\be} \rp^-\right|_\infty \int_\Om |f|^p \frac{\be}{B} d\la  \rp^{2/p} \\
	&\le \biggl(|\nabla \rho|_\infty p \lp C_{2(p-1),2}^{\be,\al} \lp \int_\Om |\nabla f|^2 \frac{\al}{A}d\la\rp^{1/2} \rp^{(p-1)} \lp \int_\Om |\nabla f|^2 \frac{\be}{B}d\la \rp^{1/2} \\
	&+  \left|\lp \Delta \rho+\nabla \rho \cdot \frac{\nabla \be}{\be} \rp^-\right|_\infty \lp C_{p,2}^{\be,\al} \lp \int_\Om |\nabla f|^2 \frac{\al}{A}d\la\rp^{1/2} \rp^{p} \biggr)^{2/p} \\
	&\le \lp p |\nabla \rho|_\infty \rp^{2/p} \lp \lp C_{2(p-1),2}^{\be,\al} \rp^2 \int_\Om |\nabla f|^2 \frac{\al}{A}d\la \rp^{(p-1)/p} \lp \int_\Om |\nabla f|^2 \frac{\be}{B}d\la \rp ^{1/p} \\
	& + \left|\lp \Delta \rho+\nabla \rho \cdot \frac{\nabla \be}{\be} \rp^-\right|_\infty^{2/p} \lp C_{p,2}^{\be,\al}\rp^2 \int_\Om |\nabla f|^2 \frac{\al}{A}d\la \\
	&=\lp \lp p |\nabla \rho|_\infty \rp^{2} \left|\frac{\be}{\al}\right|_\infty \frac{A}{B} \lp C_{2(p-1),2}^{\be,\al} \rp^{2(p-1)}\rp^{1/p} \int_\Om |\nabla f|^2 \frac{\al}{A}d\la \\
	&+ \left|\lp \Delta \rho+\nabla \rho \cdot \frac{\nabla \be}{\be} \rp^-\right|_\infty^{2/p}  \lp C_{p,2}^{\be,\al}\rp^2 \int_\Om |\nabla f|^2 \frac{\al}{A}d\la \\
	&=\lp \lp \lp  p |\nabla \rho|_\infty \rp^{2} \left|\frac{\be}{\al}\right|_\infty \frac{A}{B} \lp C_{2(p-1),2}^{\be,\al}\rp^{2(p-1)}\rp^{1/p} + \left|\lp \Delta \rho+\nabla \rho \cdot \frac{\nabla \be}{\be} \rp^-\right|_\infty ^{2/p} \lp C_{p,2}^{\be,\al}\rp^2 \rp \int_\Om |\nabla f|^2 d\laa.
\end{align*}
Here we have used the weighted Sobolev-Poincar\'{e} inequalities associated with $C_{2(p-1),2}^{\be,\al}$ and $C_{p,2}^{\be,\al}$. Note therefor that for $p\in \Bigl[2,\frac{2d-2}{d-2}\Bigr]$ it holds $p, 2(p-1) \in \Bigl[1,\frac{2d}{d-2}\Bigr]$.
\end{proof}
\smallskip
\subsection{Weighted boundary trace logarithmic Sobolev inequalities}\label{ssec:boundarytracelogsob}

\begin{lemma}[Rothaus' Lemma]\label{lem:rothaus}
Let $f:\dOm\to\mathbb{R}$ be measurable and assume that \\ $\int_{\dOm} f^2 \ln(1+f^2)d\sib <\infty$. For every $a\in\mathbb{R}$
\begin{equation*}
\ent_{\sib}\lp(f+a)^2\rp \le \ent_{\sib} \lp f^2 \rp + 2\int_{\dOm} f^2 d\sib.
\end{equation*}
\end{lemma}

\begin{lemma}\label{lem:bspent}
If $f\in C^1(\Om)$ fulfills $\int_\Om f d\sib=0$ and if there are constants $\tilde{C}_{p,2}$ such that
\begin{equation}\label{eq:cdn1}
\lp \int_{\dOm} |f|^p d\sib \rp^{2/p} \le \tilde{C}_{p,2} \int_{\Om} |\nabla f|^2 d\laa, \forall p\in\left[2, \frac{2d-2}{d-2}\right],
\end{equation}
then it holds
\begin{equation*}
\ent_{\sib}\lp f^2\rp \le \inf_{p\in\left[2,\frac{2d-2}{d-2} \right]} \frac{p}{p-2} \frac{\tilde{C}_{p,2}}{e} \int_{\Om} |\nabla f|^2 d\laa.
\end{equation*}
\end{lemma}

The proof of this Lemma is adapted from~\cite[Proposition 6.2.3]{bgl}, see also~\cite[Proposition 5.1.8]{bgl} for details.

\begin{proof}
Without loss of generality we may assume $\int_{\dOm} f^2 d\sib=1$ and define
\begin{equation*}
\phi:(0,1]\to\mathbb{R},\ \phi(r):= \ln\lp \lp \int_{\dOm} |f|^{1/r}d\sib\rp^{r} \rp.
\end{equation*}
$\phi$ is convex and $\phi'\lp \frac{1}{2} \rp= -\ent_{\sib}(f^2)$. Now for $p\in \left[ 2, \frac{2d-2}{d-2}\right]$ via the convexity of $\phi$
\begin{align*}
d\lp \phi\lp\frac{1}{2}\rp - \phi\lp\frac{1}{p}\rp\rp = d \int_{1/p}^{1/2} \phi'(s)ds &\le d \phi'\lp\frac{1}{2}\rp \lp \frac{1}{2}- \frac{1}{p}\rp  \\
\Leftrightarrow\    - \ent_{\sib}(f^2) &\ge \frac{2p}{p-2} \lp \phi\lp\frac{1}{2}\rp-\phi\lp \frac{1}{p}\rp \rp \\
\Leftrightarrow\  \phantom{-}\ent_{\sib}(f^2) &\le \frac{p}{p-2}\ln\lp \lp \int_{\dOm} |f|^p d\sib\rp^{2/p} \rp.
\end{align*}
Inserting inequality~\eqref{eq:cdn1} we obtain
\begin{equation*}
\ent_{\sib}(f^2) \le \frac{p}{p-2} \ln\lp  \tilde{C}_{p,2} \int_{\Om} |\nabla f|^2 d\laa \rp.
\end{equation*}
We define $\tilde{\phi}:(0,\infty)\to\mathbb{R}, \tilde{\phi}(r):=\frac{p}{p-2}\ln(\tilde{C}_{p,2}r)$. $\tilde{\phi}$ is concave and we may thus compute
\begin{align*}
\ent_{\sib}(f^2) &\le \tilde{\phi}\lp \int_\Om |\nabla f|^2 d\laa \rp \le \tilde{\phi}(r) + \tilde{\phi}'(r)\lp \int_\Om |\nabla f|^2 d\laa -r\rp \\
& = \tilde{\phi}'(r) \int_\Om |\nabla f|^2 d\laa + \lp\tilde{\phi}(r)-r\tilde{\phi}'(r)\rp.
\end{align*}
Choosing $r=\frac{e}{\tilde{C}_{p,2}}$ the last term vanishes and we obtain
\begin{equation*}
\ent_{\sib}(f^2) \le \tilde{\phi}'\lp \frac{e}{\tilde{C}_{p,2}} \rp \int_\Om |\nabla f|^2 d\laa = \frac{p}{p-2} \frac{\tilde{C}_{p,2}}{e} \int_\Om|\nabla f|^2 d\laa.
\end{equation*}
\end{proof}
The assumption of Lemma~\ref{lem:bspent} is phrased in terms of a family of inequalities for $p\in \left[2, \frac{2d-2}{d-2}\right]$ and the conclusion in terms of an infimum over $p\in \left[2, \frac{2d-2}{d-2}\right]$ because such a family of inequalities is what was obtained in Proposition~\ref{prop:boundaryinterior2}. However, assuming \eqref{eq:cdn1} for one fixed $p\in \lp 2, \frac{2d-2}{d-2}\right]$ instead would result in the same statement for this fixed $p$ without the infimum.\\
The following proposition provides a constant $L_{\dOm,\Om}$ fulfilling condition \eqref{eq:ent} in terms of an auxiliary function $\rho$.

\begin{proposition}
Assume that $d\ge3$. For any $\rho\in C^1(\Om)$ such that $\frac{\partial \rho}{\partial N}\vert_{\dOm}=-1$ and $\nabla \rho$ is Lipschitz continuous on $\Om$ inequality~\eqref{eq:ent} in Proposition~\ref{prop:logsobinterpol} is fulfilled with
\begin{small}
\begin{align*}
L_{\dOm,\Om} = \inf_{p\in\left[2,\frac{2d-2}{d-2} \right]}& \frac{p}{p-2} \frac{1}{e}\lp \lp \lp p |\nabla \rho|_\infty \rp^{2} \left|\frac{\be}{\al}\right|_\infty \frac{A}{B}\lp C_{2(p-1),2}^{\be,\al}\rp^{2(p-1)}\rp^{1/p} + \left|\lp \Delta \rho+\nabla \rho \cdot \frac{\nabla \be}{\be} \rp^-\right|_\infty ^{2/p} \lp C_{p,2}^{\be,\al}\rp^2 \rp\\
& + \frac{2A}{B} \left|\frac{\be}{\al}\right|_\infty \lp C_{\laa}\left|\lp \Delta \rho +\frac{\nabla \be}{\be} \cdot \nabla \rho \rp^-\right|_\infty  + 2C_{\laa}^{(1/2)}|\nabla \rho|_\infty \rp.
\end{align*}
\end{small}
\end{proposition}

\begin{proof}
Let $f\in C^1(\Om)$ then for $a:=\int_\Om f d\sib$ we define $\tilde{f}:=f-a$ and by Lemma~\ref{lem:rothaus} it holds
\begin{equation*}
\ent_{\sib}(f^2) = \ent_{\sib}((\tilde{f}+a)^2) \le \ent_{\sib}(\tilde{f}^2) + 2\int_{\dOm} \tilde{f}^2 d\sib.
\end{equation*}
Since $\int_\Om \tilde{f} d\sib=0$ the assumptions of Lemma~\ref{lem:bspent} are fulfilled due to Proposition~\ref{prop:boundaryinterior2} and we obtain
\begin{align*}
\ent_{\sib}(\tilde{f}^2)  \le \inf_{p\in\left[2,\frac{2d-2}{d-2} \right]} \frac{p}{p-2} \frac{1}{e} \biggl( &\lp \lp  p |\nabla \rho|_\infty \rp^{2} \left|\frac{\be}{\al}\right|_\infty \frac{A}{B}\lp C_{2(p-1),2}^{\be,\al}\rp^{2(p-1)}\rp^{1/p} \\
	&+ \left|\lp \Delta \rho+\nabla \rho \cdot \frac{\nabla \be}{\be} \rp^-\right|_\infty ^{2/p} \lp C_{p,2}^{\be,\al}\rp^2 \biggr) \cdot \int_{\Om} |\nabla \tilde{f}|^2 d\laa.
\end{align*}
Furthermore, by the proof of Proposition~\ref{prop:varboundary}
\begin{equation*}
\int_{\dOm} \tilde{f}^2 d\sib \le  \frac{A}{B} \left|\frac{\be}{\al}\right|_\infty \lp C_{\laa}\left|\lp \Delta \rho +\frac{\nabla \be}{\be} \cdot \nabla \rho \rp^-\right|_\infty  + 2C_{\laa}^{(1/2)}|\nabla \rho|_\infty  \rp \int_\Om |\nabla \tilde{f}|^2 d\laa.
\end{equation*}
Thus we have
\begin{small}
\begin{align*}
\ent_{\sib}(f^2)  \le \Biggl(\inf_{p\in\left[2,\frac{2d-2}{d-2} \right]}& \frac{p}{p-2} \frac{1}{e} \lp \lp \lp p |\nabla \rho|_\infty \rp^{2} \left|\frac{\be}{\al}\right|_\infty \frac{A}{B}\lp C_{2(p-1),2}^{\be,\al}\rp^{2(p-1)}\rp^{1/p} + \left|\lp \Delta \rho+\nabla \rho \cdot \frac{\nabla \be}{\be} \rp^-\right|_\infty ^{2/p} \lp C_{p,2}^{\be,\al}\rp^2 \rp \\
		& + \frac{2A}{B} \left|\frac{\be}{\al}\right|_\infty \lp C_{\laa}\left|\lp \Delta \rho +\frac{\nabla \be}{\be} \cdot \nabla \rho \rp^-\right|_\infty  + 2C_{\laa}^{(1/2)}|\nabla \rho|_\infty  \rp \Biggr) \int_{\Om} |\nabla \tilde{f}|^2 d\laa \\
	=  \Biggl( \inf_{p\in\left[2,\frac{2d-2}{d-2} \right]}& \frac{p}{p-2} \frac{1}{e}\lp \lp \lp p |\nabla \rho|_\infty \rp^{2} \left|\frac{\be}{\al}\right|_\infty \frac{A}{B}\lp C_{2(p-1),2}^{\be,\al}\rp^{2(p-1)}\rp^{1/p} + \left|\lp \Delta \rho+\nabla \rho \cdot \frac{\nabla \be}{\be} \rp^-\right|_\infty ^{2/p} \lp C_{p,2}^{\be,\al}\rp^2 \rp\\
	& + \frac{2A}{B} \left|\frac{\be}{\al}\right|_\infty \lp C_{\laa}\left|\lp \Delta \rho +\frac{\nabla \be}{\be} \cdot \nabla \rho \rp^-\right|_\infty  + 2C_{\laa}^{(1/2)}|\nabla \rho|_\infty \rp \Biggr) \int_{\Om} |\nabla f|^2 d\laa.
\end{align*}
\end{small}
\end{proof}

Combining this with Lemma~\ref{lem:rho} finally provides an explicit value for the constant $L_{\dOm,\Om}$ fulfilling condition \eqref{eq:ent} and thus allows to obtain upper bounds on $L_\mu$ and $\hat{L}_\mu$ via Propositions~\ref{prop:logsobinterpol} and~\ref{prop:logsobinterpolwobd}.

\begin{theorem}\label{thm:wbdrlogsob}
Assume that $d\ge3$. Let $k_2\in\mathbb{R}$ such that $sect \le k_2$ and $\gamma_2\in\mathbb{R}$ such that $\mathrm{I\!I}\le \gamma_2 id$. Then inequality~\eqref{eq:ent} in Proposition~\ref{prop:logsobinterpol} is fulfilled with
\begin{align*}
L_{\dOm,\Om} = &\inf_{p\in\left[2,\frac{2d-2}{d-2} \right]} \frac{p}{p-2} \frac{1}{e}\Biggl[ \lp p^{2} \left|\frac{\be}{\al}\right|_\infty \frac{A}{B}\lp C_{2(p-1),2}^{\be,\al}\rp^{2(p-1)}\rp^{1/p} \\
&+\lp\inf_{t_1\in(0,h_2^{-1}(0))} \sup_{\{\rho_{\dOm}\le t_1\}} \lp 1- \frac{\rho_{\dOm}}{t_1} \rp \lp (d-1)\frac{h_2'}{h_2}(\rho_{\dOm}) - \frac{1}{t_1-\rho_{\dOm}} - \frac{|\nabla \be|}{\be} \rp^-\rp^{2/p} \lp C_{p,2}^{\be,\al}\rp^2 \Biggr]  \\
& + \frac{2A}{B} \left|\frac{\be}{\al}\right|_\infty \biggl( C_{\laa}\inf_{t_1\in(0,h_2^{-1}(0))} \sup_{\{\rho_{\dOm}\le t_1\}} \lp 1- \frac{\rho_{\dOm}}{t_1} \rp \lp (d-1)\frac{h_2'}{h_2}(\rho_{\dOm}) - \frac{1}{t_1-\rho_{\dOm}} - \frac{|\nabla \be|}{\be} \rp^-  \\
&+ 2C_{\laa}^{(1/2)} \biggr),
\end{align*}
where $h_2$ is as defined in equation~\eqref{eq:hfctn} with $k=k_2$ and $\gamma=\gamma_2$.
\end{theorem}

\noindent\textbf{Acknowledgements} The author would like to thank Max von Renesse and David Tewodrose for their feedback
on earlier versions of the present manuscript.\\

\section*{Declarations}

\noindent\textbf{Ethical approval} Not applicable\\

\noindent\textbf{Funding} Not applicable\\

\noindent\textbf{Availability of data and materials} Not applicable

\bibliographystyle{abbrv}
\bibliography{bib}
\end{document}